\documentclass[12pt,a4paper]{amsart}
\usepackage{graphicx,amssymb,amsfonts,epsfig,amsthm,a4,amsmath,url,graphicx,enumerate,verbatim}
\usepackage[latin1]{inputenc}
\usepackage[usenames,dvipsnames]{color}

\newtheorem*{thmm}{Theorem}
\newtheorem*{propp}{Proposition}

\newtheorem{ithm}{Theorem}

\newtheorem{icor}[ithm]{Corollary}

\newtheorem{thm}{Theorem}[section]
\newtheorem{cor}[thm]{Corollary}
\newtheorem{lem}[thm]{Lemma}

\newtheorem{prop}[thm]{Proposition}
\theoremstyle{definition}
\newtheorem{defn}[thm]{Definition}

\newtheorem{que}[thm]{Question}

\theoremstyle{remark}
\newtheorem{rem}[thm]{Remark}

\numberwithin{equation}{section}

\newcommand{\Z}{\mathbf{Z}}

\newcommand{\N}{\mathbf{N}}
\newcommand{\F}{\mathbf{F}}

\newcommand{\R}{\mathbf{R}}
\newcommand{\C}{\mathcal{C}}

\newcommand{\Isom}{\text{Isom}}

\newcommand{\PSL}{\text{PSL}}

\newcommand{\eps}{\varepsilon}

\newcommand{\SL}{\textnormal{SL}}

\newcommand{\Fill}{\textnormal{Fill}}

\newcommand{\germ}{\textrm{germ}}

\makeatletter
\let\@wraptoccontribs\wraptoccontribs
\makeatother

\begin{document}

\title{Characterizing a vertex-transitive graph by a large ball}

\date{}
\author[De La Salle]{Mikael De La Salle} \address{CNRS, UMPA, ENS-Lyon\\ Lyon\\FRANCE}\email{mikael.de.la.salle@ens-lyon.fr}

\author[Tessera]{Romain Tessera} \address{Laboratoire de Math\'ematiques, Universit\'e Paris-Sud 11\\Orsay\\FRANCE}\email{romtessera@gmail.com}

\contrib[With an appendix by]{Jean-Claude Sikorav}
\address{UMPA, ENS-Lyon\\ Lyon\\FRANCE}
\email{jean-claude.sikorav@ens-lyon.fr}

\begin{abstract}It is well-known that a complete Riemannian manifold $M$ which is locally isometric to a symmetric space is covered by a symmetric space. Here we prove that a discrete version of this property  (called {\it local to global rigidity}) holds for a large class of vertex-transitive graphs, including Cayley graphs of torsion-free lattices in simple Lie groups, and Cayley graphs of torsion-free virtually nilpotent groups. By contrast, we exhibit various examples of Cayley graphs of finitely presented groups (e.g.\ $\SL_4(\Z)$) which fail to have this property, answering a question of Benjamini and Georgakopoulos. 


Answering a question of Cornulier, we also construct a continuum of pairwise non-isometric {\it large-scale simply connected} locally finite vertex-transitive graphs. This question was motivated by the fact that large-scale simply connected Cayley graphs are precisely Cayley graphs of finitely presented groups and therefore have countably many isometric classes.  \end{abstract}
\maketitle


\section{Introduction}


To illustrate the theme of this article, we start with a fact that is no doubt an easy exercise for experts in Riemannian geometry. 
\begin{propp}
Let $M$ be a homogeneous simply connected Riemannian manifold, and let $\eps>0$. Every simply connected Riemannian manifold $N$ whose balls of radius $\eps$ are isometric to the ball of radius $\eps$ in $M$ is isometric to $M$. 
\end{propp}
\begin{proof}[Sketch of proof.]
Observe that the condition on $N$ forces it to be complete, so by the main result of \cite{AS}, $N$ must be homogeneous. Let $x\in M$ and $y\in N$ and write $M=G/K$ and $N=H/L$, where $G$ and $H$ are respectively the connected components of the isometry groups of $M$ and $N$ that contain the identity, and where $K<G$ and $L<H$ are the stabilizers of $x$ and $y$ respectively (which are connected as well). The conditions on $M$ and $N$ imply that there exists an isomorphism from the Lie algebra of $G$ to the Lie algebra of $H$ that maps the Lie algebra of $K$ to the Lie algebra of $L$. Since $M$ and $N$ are simply connected, this implies that they are isometric. 
\end{proof}
This proposition says that homogeneous simply connected Riemannian manifolds have a kind of local-to-global rigidity property. Our goal is to investigate a ``large-scale" version of this property for vertex-transitive graphs that was introduced by Benjamini and Georgakopoulos. 

Throughout this paper, we equip every connected graph $X$ with its usual geodesic metric that assigns length $1$ to each edge. For brevity, we  adopt the following convention: ``a graph" means a connected, locally finite, graph without
 multiple edges and loops, and ``$x\in X$", means that $x$ is a vertex of $X$. A graph $X$ is entirely determined by the restriction of the distance to the vertex set, because there are no loops or multiple edges.  In particular the isomorphism group of the graph $X$ coincides with the isometry group of the vertex set of $X$. When $G$ is a group with a finite symmetric generating set $S$ and associated word-length $|\cdot|_S$, the Cayley graph of $G$ with respect to $S$, denoted $(G,S)$, is the graph whose vertex set is $G$ with distance $d(g,h) = |g^{-1}h|_S$. Another convention of ours is that the term ``ball'' will always stand for ``closed ball''.

\ 

Observe that  given an integer $d\geq 2$, any $d$-regular graph $X$ is covered by the $d$-regular (infinite) tree $T_d$. 
This trivial observation is a ``baby case" of the phenomenon studied in this paper.

Extending\footnote{Benjamini and Georgakopoulos
defined this only for vertex-transitive graphs.} the terminology of \cite{BsF,G}, given a graph $X$, we say that $Y$ is $R$-locally $X$ if for every vertex $y\in Y$ there exists $x\in X$ such that the ball $B_X(x,R)$, and $B_Y(y,R)$, equipped with their intrinsic geodesic metrics, are isometric. We now introduce the central notion studied in this paper. 

\begin{defn}[Local-Global Rigidity]
 Let $X$ be a graph. 
\begin{itemize}

\item{\bf (LG-rigidity)} Let $R>0$. $X$ is called {\it local to global rigid} (for short LG-rigid) {\it at scale $R$}, if every graph which is $R$-locally $X$, is covered by $X$. 

\item{\bf (SLG-rigidity)}  Let $0<r\leq R$. $X$ is called {\it strongly local to global rigid}  (SLG-rigid) {\it at scales $(r,R)$}, for some $0<r\leq R$, if  the following holds. For every graph $Y$ which is $R$-locally $X$,  and every isometry $f$ from a ball $B(x,R)$ in $X$ to a ball $B(y,R)$ in $Y$, the restriction of $f$ to $B(x,r)$ extends  to a covering from $X$ to $Y$.

 \item{\bf (USLG-rigidity)}  If in addition to the previous condition, the covering extending the partial isometry is unique, then we call $X$ {\it uniquely strongly local to global rigid} (USLG-rigid) {\it at scales $(r,R)$}.
  \end{itemize}
If there exists $R$ such that $X$ is LG-rigid at scale $R$, then we simply call $X$ LG-rigid. Similarly if for all  large enough $r$ there exists $R$ such that $X$ is SLG-rigid (resp.\ USLG-rigid) at scales $(r,R)$, then $X$ is called SLG-rigid (resp.\ USLG-rigid).  
\end{defn}

An easy compactness argument (see Proposition \ref{prop:LG<=>SLG}) shows that a graph with co-compact isometry group is SLG-rigid if (and only if) it is LG-rigid. 
It is an easy fact that if a Cayley graph is LG-rigid, then the group is finitely presented. In \cite{BsF,G}, Benjamini and Georgakopoulos asked about the converse.

\begin{que}\label{que:Cayley}
Are Cayley graphs of finitely presented groups LG-rigid? 
\end{que}

Our aim in this paper is to study this question and some generalizations. Before entering into more details, we give two simple-to-state consequences of our main results. In the positive direction we prove (see Corollary \ref{cor:virtuallyfree} and Corollary \ref{thm:lattices}, and Section \ref{section:USLG} for the terminology)
\begin{ithm}\label{thm:main_positive} Every Cayley graph of a group in one of the following families is LG-rigid:
  \begin{itemize}
  \item torsion-free irreducible lattices in real connected semisimple Lie groups with finite center.  
  \item torsion-free finitely generated groups with polynomial growth.
  \end{itemize}
\end{ithm}

In full generality, Question \ref{que:Cayley} has a negative answer.
\begin{ithm}\label{thm:main_negative} $\F_2 \times \F_2 \times \Z/2\Z$ and $\SL_4(\Z)$ have Cayley graphs which are not LG-rigid.
\end{ithm}
Notice $\SL_4(\Z)$ is a lattice in the simple Lie group $\SL_4(\R)$. This example therefore shows that the torsion-free hypothesis in the first half of Theorem \ref{thm:main_positive} is needed. We do not know whether the torsion-free hypothesis can be removed for groups with polynomial growth.
As we shall see in \S \ref{secIntroFlex}, we actually prove a stronger statement for $\F_2 \times \F_2 \times \Z/2\Z$, namely that it is not LG-rigid among vertex-transitive graphs.

We also prove (Theorem \ref{thm:discrete}) that every finitely presented group with an element of infinite order has a Cayley graph which is LG-rigid. It follows that LG-rigidity for a Cayley graph depends on the generating set. In particular LG-rigidity is not invariant under quasi-isometries.

In \cite{dlST} we see that the example of the building of $\mathrm{SL}(n,\F_p((T)))$ gives a torsion-free counterexample to question \ref{que:Cayley} for $n \geq 3$.

\subsection{Rigidity results}
Our first remark can now be reformulated as follows: $T_d$ is LG-rigid at scale $r$ for all $r>0$.
  Let us start with the following generalization. 
\begin{ithm}\label{thm:QT}
Let $X$ be a graph whose group of isometries acts cocompactly. If $X$ is quasi-isometric to a tree, then $X$ is LG-rigid.
\end{ithm}
In particular we deduce the following
\begin{icor}\label{cor:virtuallyfree}
Cayley graphs of virtually free finitely generated groups are LG-rigid.
\end{icor}

Given a graph $X$, and some  $k\in \N$, we define a 2-dimensional CW-complex $P_k(X)$ whose 1-skeleton is $X$, and whose  $2$-cells are  $m$-gons for $0\leq m\leq k$, defined by simple loops $(x_0,\ldots,x_m=x_0)$ of length $m$ in $X$, up to cyclic permutations. As a topological object, every $2$-cell is a disc attached along its boundary to a simple loop, so that the intersection of $2$ different $2$-cells belongs to the $1$-skeleton. 

\begin{defn}
Let us say that a graph $X$ is simply connected at scale $k$  (for short, $k$-simply connected) if $P_k(X)$ is simply connected. If there exists such a $k$, then we shall say that $X$ is large-scale simply connected. \end{defn}

Note that $k$-simple connectedness automatically implies $k'$-simple connectedness for any $k'\geq k$. 

Equivalently, a graph $X$ is simply connected at scale $k$ if the only graph coverings $f \colon Y \to X$ which are injective on the balls of radius $\frac k 2$ are the isometries (see Proposition \ref{prop:equivalence_of_lssc} for a proof). In particular.
\begin{prop}\label{prop:rigid implies sc}
If a vertex-transitive graph is LG-rigid at scale $R\in \N$, then it is simply connected at scale $2R$. 
\end{prop}

This is tight as shown by the standard Cayley graph  $X$  of $\Z^2$. By \cite{BE}, $X$ is LG-rigid at scale $2$. However, it is obviously not 3-simply connected as the smallest non-trivial simple loops in $X$ have length 4. 

Let $G$ be a finitely generated group and let $S$ be a finite symmetric generating subset. It is well-known that the Cayley graph $(G,S)$ is large-scale simply connected if and only if $G$ is finitely presented. More precisely $(G,S)$ is $k$-simply connected if and only if $G$ has a presentation $\langle S|R\rangle$ with relations of length at most $k$. By proposition \ref{prop:rigid implies sc}, it follows that a Cayley graph of a finitely generated group that is not finitely presented is not LG-rigid.

Let us pause here, recalling that the notion of LG-rigidity was introduced by Benjamini and Georgakopoulos in \cite{BsF,G}. The main result of \cite{G} is 
\begin{thmm}\label{thm:G}\cite{G}
One-ended planar vertex-transitive graphs are LG-rigid.
\end{thmm}
Examples of LG-rigid vertex-transitive graphs also include the standard Cayley graphs of $\Z^d$ \cite{BE}.  All these examples are now covered by the following theorem. 

\begin{ithm}\label{thm:mainIntro}
Let $X$ be a large-scale simply connected graph whose group of isometries $\Isom(X)$ is cocompact (e.g.\  $X$ is vertex-transitive). Then  $X$ is USLG-rigid if (and only if) the vertex-stabilizers of $\Isom(X)$ are finite.
\end{ithm}
In this paper we will always equip $\Isom(X)$ with the topology of pointwise convergence. The assumption that $\Isom(X)$ has finite vertex stabilizers is equivalent to $\Isom(X)$ being discrete for this topology.

Note that Theorem \ref{thm:QT} is not a consequence of  Theorem \ref{thm:mainIntro} as the automorphism group of a tree may have infinite vertex-stabilizers. 
It follows from \cite[Theorem 3.1]{Ba} that the isometry group $G$ of a one ended planar vertex-transitive graph $X$ embeds as a closed (hence discrete) subgroup of either $\PSL(2,\R)$ or of $\Isom(\R^2)$. Hence we deduce from  Theorem \ref{thm:mainIntro} that $X$ is LG-rigid, hence recovering Georgakopoulos' result.

Let us say that a finitely presented group is LG-rigid (resp. USLG-rigid) if all its Cayley graphs are LG-rigid (resp. USLG-rigid). 
Using some structural results due to Furman (for lattices) and Trofimov (for groups with polynomial growth), we obtain, as a corollary of Theorem \ref{thm:mainIntro},  
\begin{icor}\label{thm:lattices}
Under the assumption that they are torsion-free the following groups are USLG-rigid:
\begin{itemize}
\item non-virtually free irreducible lattices in semisimple connected real Lie groups with finite center (e.g.\ simple connected real Lie groups);
\item groups of polynomial growth.
\end{itemize}
\end{icor}




Before presenting the negative answer to Question \ref{que:Cayley}, let us give a useful characterization of LG-rigidity.

\begin{prop}\label{prop:characterizationrigid}
Let $k\in \N$.  Let $X$ be a $k$-simply connected graph with cocompact group of isometries. Then $X$ is LG-rigid if and only there exists $R$ such that every $k$-simply connected graph which is $R$-locally $X$ is isometric to it. 
\end{prop}
The same proof shows that $X$ is USLG-rigid if and only if for all sufficiently large $r$ there exists  $R\geq r$ such that the restriction to a ball of radius $r$ of every isometry from a ball of radius $R$ in $X$ to a ball of radius $r$ in a $R$-locally $X$ $k$-simply connected graph $Y$ extends uniquely to an isometry from $X$ to $Y$.
As an almost immediate corollary of the proof of Proposition \ref{prop:characterizationrigid}, we get 
\begin{cor}\label{cor:Y=cayley}
  Let $X$ be a Cayley graph of a finitely presented group. Then there exists $R$ such that every {\it Cayley graph} which is $R$-locally $X$ is covered by $X$.
\end{cor}
In other words, Cayley graphs of finitely presented groups are LG-rigid among Cayley graphs. We shall see later that this is not true among arbitrary graphs, not even among vertex transitive ones.

Finally we mention that in \cite{FT} an example of an infinite transitive graph $X$ was given, which is isolated among all \emph{transitive} graphs in the sense that there exists $R$ such that $X$ is the only transitive graph which is $R$-locally $X$.

\subsection{Flexibility in presence of a finite normal subgroup}\label{secIntroFlex}
We now explain the source of counterexamples to Question \ref{que:Cayley}. In this theorem $H^2(\cdot,\Z/2\Z)$ denotes the second cohomology group with values in $\Z/2\Z$ (see \S \ref{section:counterexampleGraph} for reminders on this notion).

\begin{ithm}\label{thm:2covering_notNormal2} Let $H$ be a finitely presented group and $\widetilde H$ be an extension of $H$ by $\Z/2\Z$.
  Assume that $H$ contains a finitely generated subgroup $G$ such that $H^2(G,\Z/2\Z)$ is infinite. Assume moreover that $G$ has an element of infinite order and that  the preimages in $\widetilde H$ of every element of $G$ of order $2$ have order $2$. Then $\widetilde H$ has a Cayley graph that is not LG-rigid.
\end{ithm}

The case when $\widetilde H=\Z/2\Z\times H$ is much easier (the proof is 4 pages instead of 12) and is partly based on the same idea, hence we decided to give it as a warm-up in \S \ref{section:counterexampleGraph}.
Moreover this case does not require any assumption on $G$.

Requiring that $G$ is normal and that $H$ is a semidirect product of $G$ by $H/G$, we can get a stronger form of non LG-rigidity, where the graphs negating the LG-rigidity are \emph{transitive} graphs: 

\begin{ithm}\label{thm:2covering_Normal2}
Let $H$ be a finitely presented group and $\widetilde H$ be an extension of $H$ by $\Z/2\Z$. Assume that $H$ is isomorphic to a semi-direct product $G\rtimes  Q$ such that $G$ is finitely generated and $H^2(G,\Z/2\Z)$ is infinite. Assume moreover that $G$ has an element of infinite order and that the preimages in $\widetilde H$ of every element of $G$ of order $2$ have order $2$. Then there is a Cayley graph $X$ of $\tilde H$ that is not LG-rigid among transitive graphs. More precisely, for every $R\ge 1$, there exists a family with the cardinality of the continuum $(X_{i})_{i\in I}$ of large-scale simply connected vertex-transitive graphs that are pairwise non-isometric such that for every $i\in I$,
\begin{itemize}
\item[(i)] $X_i$ is $R$-locally $X$ and $4$-bilipschitz equivalent to $X$;
\item[(ii)] The isometry group of $X_i$ is an extension of a discrete group by $(\Z/2\Z)^\N$.
\end{itemize}
\end{ithm}

\begin{rem}\label{rem:NonLGrigid}
Being $4$-bilipschitz equivalent to $X$, $X_i$ is $k$-simply connected where $k$ does not depend neither of $R$ nor of $i$ (by Theorem \ref{thm:QIinvariance}). It follows from (i) and Proposition \ref{prop:characterizationrigid} that $X$ is not LG-rigid (in a very strong sense).   
\end{rem}
\begin{rem}\label{rem:torsionfree}
The assumption that $G$ has an element of infinite order is conjecturally not needed. It is a minor technical assumption that allows us to use a variant of Theorem \ref{thm:discrete}. Without it we can prove the Theorem for a Cayley graph of $\Z/N\Z \times H$ for some $N$.
\end{rem}

An explicit example for which  Theorem \ref{thm:2covering_Normal2} applies is $H=\F_2 \times \F_2$, and $G$ the kernel of the homomorphism $\F_2 \times \F_2 \to \Z$ which sends each generator of each copy of the free group $\F_2$ to $1$. Alternatively, one could also take for $H$ a product of two surface groups of genus at least $2$. This fact is probably known, and was explained to us by Jean-Claude Sikorav. We could not find a reference in the literature and instead provide a proof in Appendix \ref{sec:appendix}. Note that Theorem \ref{thm:main_negative} is a consequence
of Appendix \ref{sec:appendix} and Theorem \ref{thm:2covering_Normal2}.

In particular, we deduce that Theorem \ref{thm:2covering_notNormal2} applies to any finitely presented group $H$ containing $\F_2 \times \F_2$. For example to $\PSL(4,\Z)$, which contains a subgroup isomorphic to $\F_2 \times \F_2$ because $\PSL(2,\Z)$ contains a subgroup isomorphic to $\F_2$. So Theorem \ref{thm:main_negative} is a consequence of Appendix \ref{sec:appendix} and Theorem \ref{thm:2covering_Normal2}.

We end this discussion with the following question.

\begin{que}
Among lattices in semisimple Lie groups, which ones are LG-rigid? For instance is $\PSL(3,\Z)$ LG-rigid?
\end{que}

Note that since large-scale simply connected Cayley graphs are precisely Cayley graphs of finitely presented groups, there are countably many such isometry classes of such graphs. Cornulier asked whether there exist uncountably many isometry classes of large-scale simply connected vertex-transitive graphs. The previous theorem answers positively this question. It would be interesting to know whether there exist uncountably many quasi-isometry classes of  large-scale simply connected vertex-transitive graphs. Observe that this is not answered by our result.

\subsection{Cayley graphs with discrete isometry group}

We conjecture that every finitely generated group has a Cayley graph (without multiple edges) with discrete isometry group. 
In the general case the closest to this conjecture that we can get is the following theorem.
\begin{ithm}\label{prop:discreteIsometriesCayleyGraph} Let $G$ be a finitely generated group. There is a finite cyclic group $F$ and a Cayley graph of $G \times F$ with discrete isometry group.
\end{ithm}

More involved is the following result, where we prove the conjecture in the case when the group has an element of infinite order. A variant of this result plays a crucial role in the proof of Theorem \ref{thm:2covering_Normal2}.

\begin{ithm}\label{thm:discrete}
Every finitely generated group $G$ with an element of infinite order admits a Cayley graph $(G,S)$ with discrete group of isometries. If in addition $G$ is finitely presented, we deduce that $(G,S)$ is USLG-rigid. 
\end{ithm}

Let us mention the following consequence, which answers a question by Georgakopoulos \cite[Problem 1.2]{G} under the additional assumption that $G$ has an element of infinite order.
\begin{icor}\label{cor:residualfinite} Let $G$ be a finitely presented group with an element of infinite order. If all the Cayley graphs $X$ of $G$ admit a sequence $(Y_n)_n$ of finite graphs which are $n$-locally $X$, then $G$ is residually finite.
\end{icor}

\subsection{From graphs to cocompact geodesic metric spaces}

Finally, one may wonder whether Theorem \ref{thm:mainIntro} can be generalized to more general geodesic metric spaces. The following construction provides serious limitations to this hope.  
\begin{ithm}\label{thm:graph}
There exists a  metric space $X$ with the following properties.
\begin{itemize}
\item[(i)] $X$ is proper, geodesic, and contractible.
\item[(ii)] $\Isom(X)\simeq \Z$ (in particular it has trivial point stabilizers).
\item[(iii)] $\Isom(X)$ is cocompact. More precisely, there exists $x\in X$ such that $\Isom(X)\cdot B(x,1)=X$.
\item[(iv)] For every $R$, there exists a continuum of pairwise non isometric metric spaces $Y_R$ which are $R$-locally $X$ and satisfying (i), (ii) and (iii).
\item[(v)] For every $R$, there exists a continuum of pairwise non isometric metric spaces $Y_R'$ which are $R$-locally $X$ but have a trivial isometry group.
\item[(vi)] For every $R$, there exists a continuum of pairwise non isometric metric spaces $Y_R''$ which are $R$-locally $X$ and have an uncountable isometry group (cocompact or not).
\end{itemize}
\end{ithm}

\subsection{Application to Benjamini-Schramm convergence of finite random graphs}

\begin{defn}
A sequence of graphs $(Y_n)$  is called asymptotically $k$-simply connected if for every $l\in \N$, there exists $L\in \N$ such that for $n$ large enough, every loop of length at most $l$ in $Y_n$ bounds a disc of diameter at most $L$ in $P_k(Y_n)$.
\end{defn}

As a corollary of Proposition \ref{prop:characterizationrigid}, we have the following result which says that for a sequence of asymptotically $k$-simply connected graphs to converge in the Benjamini-Schramm topology \cite{BS} to a $k$-simply-connected LG-rigid graph, it is enough that the balls of a fixed radius converge.

This corollary was suggested by Itai Benjamini. 
\begin{prop}\label{cor:BS_convergence}
 Let $X$ be a $k$-simply-connected LG-rigid graph with cocompact group of isometries. There exists $R$ such that the following holds. If $(Y_n)$ is a sequence of finite graphs such that a proportion $1-o(1)$ of the balls of radius $R$ in $Y_n$ are isometric to a ball in $X$, and such that $Y_n$ is asymptotically $k$-simply connected, then for every $R'$, a proportion $1-o(1)$ of the balls of radius $R'$ in $Y_n$ are isometric to a ball in $X$.
\end{prop}
Proposition \ref{cor:BS_convergence} is proved in \S \ref{sec:proofRandom}.

\subsection{Explicit bounds on the LG-rigidity radius?}

The main weakness of all our rigidity results is that they do not provide any explicit estimate for the LG-rigidity radius, i.e.\ the value of $R$ appearing the definition of LG-rigidity. It would be especially interesting to get such estimates for torsion-free lattices in simple Lie groups. In \cite{BE}, it is proved that for  the standard Cayley graph of $\Z^d$, $R=2$ for $d=2$ and $R=3$ for $d\geq 3$. It is interesting to note that this gives a uniform bound for all $d$. More generally, one may ask whether there is a uniform upper bound $R(l)$ for the radius of LG-rigidity that holds for all (torsion-free?) Cayley graphs of $l$-step nilpotent groups. Note that $R(l)$ necessary tends to infinity as the standard Cayley graph of the free nilpotent group  of step $l$ and rank $r$ converges to the Cayley graph of the free group of rank $r$ as $l\to \infty$. This would already be new for the abelian case $l=1$. This  question also makes sense for all torsion free lattices of a given simple Lie group.

Note that by contrast, given a finitely generated group $G$ containing a torsion-free element, and a finite generating set $S$, the proof of Theorem \ref{thm:discrete} gives an constructive way of modifying the Cayley graph $(G,S)$ to obtain a Cayley graph of $G$ with a discrete group of isometries: let's call it the ``rigidified Cayley graph". It would be interesting, starting with an explicit Cayley graph of --say-- $SL_n(\Z)$, to find an explicit bound for the radius of LG-rigidity of its rigidified Cayley graph.

\subsection*{Organization of the paper} The paper is organized as follows. Section \ref{section:preliminaries} and \ref{section:large-scaleSC} contain preliminaries on large scale simple connectedness and the proofs of Propositions  \ref{prop:rigid implies sc}, \ref{prop:characterizationrigid} and \ref{cor:BS_convergence}. Section \ref{section:QT} and \ref{section:USLG} contain our rigidity results for quasi-trees (Theorem \ref{thm:QT}) and graphs with discrete isometry groups (Theorem \ref{thm:mainIntro}) respectively. In Section \ref{section:discrete}, we prove Corollary \ref{thm:lattices}. 
Section \ref{section:counterexampleGraph} contains the proof of a particular case of Theorem \ref{thm:2covering_notNormal2} (namely for direct products); Section contains the proof of the general case, as well as of Theorem  \ref{thm:2covering_Normal2}, using the content of Section \ref{section:discreteisom}. Theorems \ref{prop:discreteIsometriesCayleyGraph} and \ref{thm:discrete} are proved in Section \ref{section:discreteisom}. Finally, the proof of Theorem \ref{thm:graph} is provided in Section  \ref{section:counterexamplemetric}.

\subsection*{Acknowledgement} We are grateful to David Ellis whose questions helped us clarifying various statements and parts of the proofs. We also thank very much the anonymous referee for his careful reading and his numerous and precise comments and suggestions. The research of both authors was supported by the ANR project GAMME (ANR-14-CE25-0004). The first author's research was also supported by the ANR project AGIRA (ANR-16-CE40-0022).

\section{Preliminaries about $k$-simple connectedness}\label{section:preliminaries}

Except for the following paragraph, dealing with the quasi-isometry invariance of large-scale simple connectedness, the following material is not needed in the rest of the paper, but we include it in order to advertise the naturality of the $2$-complex $P_k(X)$ for vertex-transitive graphs.

\subsection{Equivalent definitions of $k$-simple connectedness}


\begin{prop}\label{prop:equivalence_of_lssc} For a graph $X$, the following are equivalent.
  
  \begin{itemize}
  \item $P_k(X)$ is simply connected.
  \item The only graph coverings $f \colon Y \to X$ which are injective on the balls of radius $\frac k 2$ are the isometries. 
  \end{itemize}
  \end{prop}
\begin{proof}
Assume that $P_k(X)$ is simply connected. Let $f \colon Y \to X$ be a graph covering which is injective on balls of radius $\frac k 2$. Since every simple loop of length $\leq k$ is contained in a (closed) ball of radius $\frac k 2$, $f$ induces a covering $P_k(Y) \to P_k(X)$. Since $P_k(X)$ is simply connected, the covering $P_k(Y) \to P_k(X)$ is trivial (bijective), and in particular $f$ is trivial.

Conversely, assume that the only graph coverings $f \colon Y \to X$ which are injective on the balls of radius $\frac k 2$ are the isometries. To prove that $P_k(X)$ is simply connected, we prove that if $Y$ is the $1$-skeleton of the universal cover of $P_k(X)$, then the covering $Y \to X$ is trivial. We actually show that the covering $Y \to X$ is injective on balls of radius $\frac k 2$: indeed, if $x,y$ are two points in a ball of radius $\frac k 2$ of $Y$ which map to the same point in $X$, then we can consider a path of length $\leq k$ joining $x$ and $y$ in $Y$. Its image in $X$ is a simple loop of length $\leq k$, and so by definition of $P_k(X)$ it is homotopic to the trivial loop in $P_k(X)$. This implies that the path we started with in $Y$ is a simple loop, and hence $x=y$. This proves that $f$ is injective on balls of radius $\frac k 2$.
\end{proof}

\subsection{Invariance under quasi-isometry}
Given two constants $C\geq 1$ and $K\geq 0$, a map $f:X\to Y$ between two metric spaces is a $(C,K)$-quasi-isometry if every $y\in Y$ lies at distance $\leq K$ from a point of $f(X)$, and if for all $x,x'\in X$, 
$$C^{-1}d_X(x,x')-K\leq d_Y(f(x),f(x'))\leq Cd_X(x,x')+K.$$

\begin{thm}\label{thm:QIinvariance}
Let $k\in \N^*$, $C\geq 1$, $K\geq 0$ and let $X$ be a $k$-simply connected graph. Then there exists $k'\in \N^*$ such that every graph $Y$ such that there exists a $(C,K)$-quasi-isometry from $X$ to $Y$, is $k'$-simply connected. 
\end{thm}
\begin{proof}
Since this is well-known, we only sketch its proof (which roughly follows the same lines as the proof of \cite[Proposition 6.C.4]{CH}). The strategy roughly consists in showing that simple $k$-connectedness is equivalent to a property that is defined in terms of the metric space $X$, and which will obviously be invariant under quasi-isometries (up to changing $k$).

In the sequel, a path $\gamma$ joining two vertices $x$ to $x'$ in a graph $X$ is a sequence of vertices $(x=\gamma_0,\ldots,\gamma_n=x')$ such that $\gamma_i$ and $\gamma_{i+1}$ are adjacent for all $0\leq i<n$. We consider the equivalence relation $\sim_{k,x,x'}$ between such paths $\gamma=(\gamma_0,\ldots,\gamma_n)$ and $\gamma=(\gamma'_0,\ldots,\gamma'_{n'})$ generated by $\gamma\sim_{k,x,x'}\gamma'$ if they ``differ by at most one 2-cell", i.e.\ if $n=j_1+j_2+j_3$, $n'=j_1+j'_2+j_3$ such that
\begin{itemize}
\item 
$\gamma_i=\gamma'_i$ for all $i\leq j_1$; 
\item $\gamma_{j_1+j_2+i}=\gamma'_{j_1+j_2'+i}$  for all $i\leq j_3$;
\item $j_2+j_2'\leq k$.
\end{itemize}
We leave as an exercice the fact that $P_k(X)$ is simply connected if and only if for all $x,x'$, the equivalence relation $\simeq_{k,x,x'}$ has a single equivalence class. Note that this reformulation allows to work directly in the graph $X$. But it still has the disadvantage that it is defined in terms of combinatorial paths in $X$, based on the notion of adjacent vertices (which does not behave well under quasi-isometries). In order to solve this issue, but at the cost of changing $k$, we now define a more flexible notion of paths in $X$:
given a constant $C>0$, we define a $C$-path in $X$ from $x$ to $x'$ to be a sequence $x=\eta_0, \ldots, \eta_n=x'$ such that $d(\eta_i,\eta_{i+1})\leq C$ for all $0\leq i<n$. Given some $L>0$, we define the equivalence relation $\sim_{C,L,x,x'}$ between $C$-paths joining $x$ to $x'$ to be the equivalence relation generated by the relation 
$\eta \sim_{C,L,x,x'} \eta'$ if there exist non-negative integers $j_1,j_2,j_2'$ and $j_3$ such that, letting  $n$ and $n'$ be  the lengths of respectively $\eta$ and $\eta'$, we have 
\begin{itemize}
\item $n=j_1+j_2+j_3$ and $n'=j_1+j'_2+j_3$;
\item 
$\eta_i=\eta'_i$ for all $i\leq j_1$; 
\item $\eta_{j_1+j_2+i}=\eta'_{j_1+j_2'+i}$ for all $i\leq j_3$; 
\item $j_2+j_2'\leq L$.
\end{itemize}
It is easy to see that  if $X$ is $k$-simply connected, then for every $C$, there exists $L$ only depending on $C$ and $k$ such that for all $x,x'$ the equivalence relation $\sim_{C,L,x,x'}$ has a single equivalence class. Conversely, if for some $C\geq 1$ and $L$, the equivalence relation $\sim_{C,L,x,x'}$ has a single equivalence class for all $x,x'\in X$, then $X$ is $k'$-simply connected for some $k'$ only depending on $C$ and $L$. 
Now the latter condition is designed to be invariant under quasi-isometries, so we are done.
\end{proof}

\subsection{Cayley-Abels graph}\label{sec:Cayley-Abels}

Let $X$ be a locally finite vertex-transitive graph, and let $G$ be its group of isometries. Recall that $G$ is locally compact for the compact open topology.  Given some vertex $v_0$, denote by $K$ the stabilizer of $v_0$ in $G$: this is a compact open subgroup. Let $S$ be the subset of $G$ sending $v_0$ to its neighbors. One checks that $S$ is a compact open symmetric generating subset of $G$ and that $S$ is bi-$K$-invariant: $S=KSK$.  It follows that the Cayley graph $(G,S)$ is invariant under the action of $K$ by right translations, and that $X$ naturally identifies to the quotient of $(G,S)$  under this action. Conversely, given a totally disconnected, compactly generated, locally compact group $G$, one can construct a locally finite graph on which $G$ acts continuously, properly,  and vertex-transitively. To do so, just pick a compact open subgroup $K$ and a compact symmetric generating set $T$, define $S=KTK$ and consider as above the quotient of the  Cayley graph $(G,S)$ by the action of $K$ by right translations (note that the vertex set is just $G/K$). 
This construction, known as the  Cayley-Abels graph $(G,K,S)$ of $G$ with respect to $S$ and $K$ generalizes the more classical notion of Cayley graph, which corresponds to the case where $K=1$ (and $G$ is discrete).

\subsection{Cayley-Abels $2$-complex}

We start by recalling some basic facts about group presentation and presentation complex for abstract groups (not necessarily finitely generated).  
Let $G$ be a group, and let $S$ be a symmetric generating subset of $G$. We consider the Cayley graph $(G,S)$ as a graph whose edges are labelled by elements of $S$.
Let $R$ be a subset of the kernel of the epimorphism $\phi: F_S\to G$. Consider the polygonal 2-complex $X=X(G,S,R)$, whose 1-skeleton is the Cayley graph $(G,S)$, and where a $k$-gon is attached to every $k$-loop labeled by an element of $R$. It is well-known that $X$ is simply-connected if and only if the normal subgroup generated by $R$ is $\ker \phi$. In this case,  $\langle S; R \rangle$ defines a presentation of $G$, and $X$ is called the Cayley 2-complex associated to this presentation.

The proof of this statement extends without change to the following slightly more general setting: assume that $K$ is a subgroup of $G$ such that $S=KSK$, and consider the Cayley-Abels graph $(G,K,S)$. Let $v_0$ be the vertex corresponding to $K$ in $(G,K,S)$. 

 Consider the polygonal 2-complex $X=X(G,S,R)$, whose 1-skeleton is the Cayley-Abels graph $(G,K,S)$, and where a $k$-gon is attached to every $k$-loop which is obtained as the projection in $X$ of a $k$-loop labelled by some element of $R$ in $(G,S)$. Once again, one checks $X$ is simply-connected if and only if $R$ generates $\ker \phi$. In this case,  $\langle S; R \rangle$ defines a presentation of $G$, and we call $X$ the Cayley-Abels 2-complex associated to this presentation.

\subsection{Compact presentability and $k$-simple connectedness}

Recall that a locally compact group is compactly presentable if it admits a presentation $\langle S; R \rangle$, where $S$ is a compact generating subset of $G$, and $R$ is a set of words in $S$ of length bounded by some constant $k$. Now let $K$ be a compact open subgroup and let $S$ be a such that $S=KSK$. We deduce from the previous paragraph that the morphism $\langle S; R \rangle\to G$ is an isomorphism if and only if the Cayley-Abels graph $(G,K,S)$ is $k$-simply connected.

\section{Large-scale simple connectedness and LG-rigidity}\label{section:large-scaleSC}

This section is dedicated to the proofs of the rather straightforward Propositions \ref{prop:rigid implies sc} and \ref{prop:characterizationrigid}. It can be skipped by the reader only interested in our main results.

\subsection{Proof of Proposition \ref{prop:rigid implies sc} }

\begin{lem}
Let $X$ and $Z$ be two graphs, and let $R\geq 1$. Assume that $X$ is vertex-transitive,  that $Z$ is $R$-locally $X$, and that $p:Z\to X$ is a covering map. Then $p$ is an isometry in restriction to balls of radius $R$.
\end{lem}
\begin{proof}
Since $p$ is a covering map, for all $z\in Z$, $p(B(z,R))=B(p(z),R)$. Since $Z$ is $R$-locally $X$, $B(z,R)$ is isometric to some ball of radius $R$ in $X$, and hence (since $X$ is vertex-transitive) to $B(p(z),R)$. In particular $B(z,R)$ and $B(p(z),R)$ have same cardinality, which implies that $p$ must be injective in restriction to $B(z,R)$. Hence we are done.
\end{proof}
We obtain as an immediate corollary: 
\begin{cor}\label{cor:selfcover}
Let $X$ be a vertex-transitive graph. Every self-covering map $p:X\to X$ is an automorphism.
\end{cor}

Let us turn to the proof of the proposition. Let $p \colon Y \to X$ be a covering which is injective on balls of radius $R$. In particular $Y$ is $R$-locally $X$. Hence we have a covering map $q:X\to Y$. By Corollary \ref{cor:selfcover}, $q\circ p$ is an automorphism, implying that $p$ is injective and therefore is a graph isomorphism. By Proposition \ref{prop:equivalence_of_lssc} $X$ is $2R$-simply connected, so we are done.


\subsection{Proofs of Proposition \ref{prop:characterizationrigid}, Corollary \ref{cor:Y=cayley} and Proposition \ref{cor:BS_convergence}}\label{sec:proofRandom}

Since $X$ has a cocompact group of isometries, there are only finitely many orbits of vertices. Therefore, since  $P_k(X)$ is simply connected, for all $R_1\in \N$ there exists $R_2\in \N$ such that every loop in $X$ based at some vertex $x$ and contained in $B(x,R_1)$ can be filled in inside $P_k(B(x,R_2))\subset P_k(X).$ It turns out that Proposition \ref{prop:characterizationrigid}  can be derived from a more general statement, which requires the following definition (which is a variant of Gromov's filling function \cite{Gr}).
\begin{defn}
We define the $k$-Filling function of a graph $X$ as follows: for every $R_1>0$, $\Fill_X^k(R_1)$ is the infimum over all $R_2 \geq R_1$ such that  every loop based at some vertex $x\in X$ and contained in $B(x,R_1)$ is homotopic to the
constant loop inside $P_k(B(x,R_2))$.
\end{defn}
Note that even if $X$ is $k$-simply connected but does not have a cocompact isometry group, $\Fill_X^k$ can  take infinite values\footnote{It is easy to come up with an example, considering a chain of bottle-shaped graphs with bottleneck size tending to infinity at lower speed than the core of the bottle. We leave the details to the reader.}. Proposition \ref{prop:characterizationrigid}  is now a corollary of
\begin{prop}\label{prop:characterizationrigidbis}
Let $X$ be a $k$-simply connected graph with finite $k$-Filling function.  Then $X$ is LG-rigid if and only if there exists $R$ such that every $k$-simply connected graph which is $R$-locally $X$ is isometric to it. Similarly, $X$ is (U)SLG-rigid if for all large enough $r>0$ there exists $R\geq r$ such that the following holds. For every $k$-simply connected graph $Y$ which is $R$-locally $X$,  and every isometry $f$ from a ball $B(x,R)$ in $X$ to a ball $B(y,R)$ in $Y$, the restriction of $f$ to $B(x,r)$ extends (uniquely) to an isometry from $X$ to $Y$.
\end{prop}
For this proposition we shall use the following notion, already used in Proposition \ref{prop:equivalence_of_lssc}. If $X$ is a graph and $k \in \N$, the \emph{$k$-universal cover of $X$} is the $1$-skeleton of the universal cover of $P_k(X)$. For example, if $X$ is a Cayley graph $(G,S)$, then the $k$-universal cover of $X$ is the Cayley graph $(\widetilde G,S)$ where $\widetilde G$ is given by the presentation $\langle S|R\rangle$, with $R$ the words of length at most $k$ that are trivial in $G$.
\begin{lem}\label{lem:Z is k-simply connected}
Let $X$ be a graph and $k\in \N$. The $k$-universal cover of $X$ is $k$-simply connected.
\end{lem}
\begin{proof}
Let $Q$ be the universal cover of $P_k(X)$ and $Z$ its $1$-skeleton, \emph{i.e.} the $k$-universal cover of $X$. Observe that the 2-cells of $Q$ consist of $m$-gons for some $m\leq k$, that are attached to simple loops of length $m$ in $Z$. Hence $P_{k}(Z)$ is obtained from $Q$ by possibly attaching more $2$-cells. It follows that $P_{k}(Z)$ is simply connected.
\end{proof}

We shall need the following lemma as well.
\begin{lem}\label{lem:covering_injective} Let $X$ be $k$-simply connected graph with finite $k$-Filling function. For every $R_1>0$, there exists $R_2$ such that if a graph is $R_2$-locally $X$ then its $k$-universal cover is $R_1$-locally $X$.
\end{lem}
\begin{proof} Let $R_1>0$. Take $R_2>\Fill_X^k(R_1)$, and assume that a graph $Y$ is $R_2$-locally $X$. Let $p \colon Z \to Y$ be its $k$-universal cover. We claim  that $p$ is injective in restriction to balls of radius $R_1$: this implies that $Z$ is $R_1$-locally $Y$, and hence $R_1$-locally $X$ because $R_2 \geq R_1$, and we are done. Indeed, let $y\in Y$, and $z\in Z$ such that $p(z)=y$. Now let $z_1$ and $z_2$ be two elements of $B(z,R_1)$ such that $p(z_1)=p(z_2)=y'$. We let $\gamma_1$ and $\gamma_2$ be two geodesic paths joining $z$ respectively to $z_1$ and $z_2$, and we let $\bar{\gamma}_1$, and  $\bar{\gamma}_2$ be the corresponding paths in $Y$, both joining $y$ to $y'$.
The concatenation of $\bar{\gamma}_1$ with the inverse of $\bar{\gamma}_2$ defines a loop $\alpha$ based at $y$ and contained in $B(y,R_1)$. But since $Y$ is $R_2$-locally $X$, $\alpha$ can be filled in inside $P_k(B(y,R_2))$, and in particular inside $P_k(Y)$. From the assumption that $p\colon z \to Y$ is the $k$-universal cover, we deduce that $z_1=z_2$. Hence the claim is proved. 
\end{proof} 


\begin{proof}[Proof of Proposition \ref{prop:characterizationrigidbis}] We shall only prove the first statement, the two other ones being very similar. Let us assume first that $X$ is $R$-LG rigid for $R\geq k/2$, and let $Y$ be $k$-simply connected and $R$-locally $X$. Then $Y$ is covered by $X$, and it follows from Proposition  \ref{prop:equivalence_of_lssc} that this covering map is an isometry. This proves the first implication.

Let us turn to the (more subtle) converse implication. Assume that $X$ is $k$-simply connected, and that there exists $R$ such that the following holds: every $k$-simply connected graph which is $R$-locally $X$ is isometric to it.
Let $R_1=R$, and let $R_2$ as in Lemma \ref{lem:covering_injective}. If $Y$ is $R_2$-locally $X$, then its $k$-universal cover is $R$-locally $X$, and hence is isometric to $X$. This gives a covering $X \to Y$ and concludes the proof.
\end{proof}


\begin{lem}\label{lem:asymptSC}
Let $(Y_n,o_n)$ be a sequence of graphs locally converging to some graph $(Y,o)$: i.e.\ such that for all $r$, the ball $B(o_n,r)$ is isometric to $B(o,r)$ for all $n$ large enough. Assume that the sequence $Y_n$ is asymptotically $k$-simply connected, then $Y$ is $k$-simply connected.
\end{lem}
\begin{proof}
Let $\gamma$ be a loop of length $l$ in $Y$. Let $L$ be as in the definition of asymptotic $k$-simple connectedness. The loop $\gamma$ is contained in $B(o,r)$ for some $r$. Now for $n$ large enough the ball of radius $B(o_n,2L)$ in $P_k(Y_n)$ is isometric to $B(o,2L)$ in $P_k(Y)$, hence we deduce that there is a disc (of diameter at most $L$) bounded by $\gamma$ in $P_k(Y)$.  
\end{proof}

\medskip

Let us prove Corollary \ref{cor:Y=cayley}. Let $X=(G,S)$ be the Cayley graph of a finitely presented group. Let $k \in \N$ be such that $X$ is $k$-simply connected. Observe that the number of isometry classes of Cayley graphs $Z=(H,S')$ where $H$ is given by a presentation $\langle S',R\rangle$ with $|S|=|S'|$ and with relations of length at most $k$ is bounded by a function of $|S|$ and $k$. Hence, it follows from an easy compactness argument that for $R_1$ large enough, if such a $Z$ is $R_1$-locally $X$, it is isometric to $X$.

Let $R_2>0$ be given by Lemma \ref{lem:covering_injective} for $X$. Let $Y=(H_0,S')$ be a Cayley graph $R_2$-locally $X$. Then its $k$-universal cover is $R_1$-locally $X$, and is the Cayley graph $(H,S')$ for the group $H$ given by the presentation $\langle S' | R\rangle$ where $R$ is the set of words of length less than $k$ that are trivial in $H_0$. It is therefore isometric to $X$. This implies that $X$ covers $Y$ and proves Corollary \ref{cor:Y=cayley}.

\medskip

We now prove Proposition \ref{cor:BS_convergence}. By Proposition \ref{prop:characterizationrigid} there exists $R \geq 2$ such that every $k$-simply connected graph which is $R$-locally $X$ is isometric to $X$. We prove Proposition \ref{cor:BS_convergence} for this $R$. Let $B_n \subset Y_n$ denote the set of (bad) vertices $y \in Y_n$ such that $B(y,R)$ is not isometric to a ball in $X$. If $d$ is the maximum degree of a vertex in $X$, then the set of points at distance $1$ from $B_n$ has size at most $d|B_n|$, because every such point has a neighbour in $B_n$, and (because $R\geq 2$) this neighbour has degree at most $d$. By the same argument, the cardinality of the set of points in $Y_n$ at distance at most $r$ from $B_n$ is bounded above by $|B_n| (1+r d^r)$. In particular there exists a sequence $r_n$ going to infinity such that, for a proportion $1-o(1)$ of the vertices in $y \in Y_n$, $B(y,r_n)$ does not intersect $B_n$. Denote by $C_n$ the set of all such vertices. We claim that for every $R'>0$, there exists $n(R')$ such that $B(y,R')$ is isometric to a ball in $X$ for every $y \in C_n$ and every $n \geq n(R')$. If this was not true, we could find a sequence $n_\alpha$ going to infinity, a vertex $y_\alpha \in C_{n_\alpha}$ such that $B(y_\alpha,R')$ is constant and different from a ball in $X$. By extracting a subsequence, we can even assume that $B(y_\alpha,R'')$ converges for every $R'' > R'$, to the ball of radius $R''$ around $y$ of some graph that we denote by $Y$. By Lemma \ref{lem:asymptSC}, $Y$ is $k$-simply connected as a limit of asymptotically $k$-simply connected graphs. Also, $Y$ is $R$-locally $X$ because every ball $B(y',R) \subset Y$ is isometric to a ball $B(y'_\alpha,R)$ around a point $y'_\alpha$ at distance at most $d(y,y')$ from $y_\alpha$ for infinitely many $\alpha$'s; in particular, taking $\alpha$ such that $r_{n_\alpha} \geq d(y,y')$, this ball is isometric to a ball of radius $R$ in $X$ by the definition of $C_n$. Therefore $Y$ is isometric to $X$, and in particular $B(y,R')$ (which coincides with $B(y_\alpha,R')$ for every $\alpha$) is isometric to a ball in $X$. This is a contradiction, and proves the Proposition.

\medskip

We end this section with the proof of a result stated in the introduction.

  \begin{prop}\label{prop:LG<=>SLG} Let $X$ be a graph with a cocompact group of isometries. If $X$ is LG-rigid, then it is SLG-rigid: i.e.\ for every $r>0$ there is $R>0$ such that, if $Y$ is R-locally $X$, for every isometry $f$ from a ball $B(x,R)$ in $X$ to a ball $B(y,R)$ in $Y$, the restriction of $f$ to $B(x,r)$ extends to a covering from $X$ to $Y$.
  \end{prop}
  \begin{proof}
    Take $X$ as in the Proposition and $r>0$. By Proposition \ref{prop:rigid implies sc} there is $k$ such that $X$ is $k$-simply connected. Since $X$ has a cocompact isometry group, there is $R_1$ such that for every $x\in X$, the restriction to $B_X(x,r)$ of an isometry $f\colon B_X(x,R_1)\to X$ coincides with the restriction of an element of $\Isom(X)$ (see Lemma \ref{lem:cocompact}). By enlarging $R_1$, we can assume that $X$ is LG-rigid at scale $R_1$. Let $R_2$ be given by Lemma \ref{lem:covering_injective} for $X$. We prove the conclusion of the Proposition with $R=R_2$.

Let $Y$ be a graph $R_2$-locally $X$ and $f\colon B(x,R_2) \to B(y,R_2)$ be an isometry. Consider $Z$, the $k$-universal cover of $Y$. By Lemma \ref{lem:covering_injective} $Z$ is $R_1$-locally $X$ and $f$ lifts to an isometry $\tilde f \colon B(x,R_1) \to B(z,R_1)$. Moreover as in the proof of Proposition \ref{prop:characterizationrigidbis}, $Z$ is isometric to $X$. By our choice of $R_1$, the restriction of $\tilde f$ to $B(x,r)$ extends to an isometry $X \to Z$. The composition with the covering map $Z \to Y$ gives the required covering extending the restriction of $f$ to $B(x,r)$.
  \end{proof}

\section{The case of quasi-trees: proof of Theorem \ref{thm:QT}}\label{section:QT}

We start with an elementary general Lemma.
\begin{lem}\label{lem:cocompact} Let $X$ be a graph with cocompact isometry group. Given some $r\geq 0$, there exists $r_2$ such that~:
\begin{itemize} 
\item for every $x\in X$, the restriction to $B_X(x,r)$ of an isometry $f\colon B_X(x,r_2)\to X$ coincides with the restriction of an element of $\Isom(X)$.
\item if $R > r_2$ and if $Y$ is $R$-locally $X$ and $x\in X$, then the restriction to $B_X(x,r)$ of an isometry $f\colon B_X(x,r_2)\to Y$ coincides with the restriction of an isometry $B_X(x,R)\to Y$.
\item if $R>r_2+1$ and if $Y$ is $R$-locally $X$, then every covering $X \to Y$ is injective on balls of radius $R$.
\end{itemize}
\end{lem}
\begin{proof} By the assumption that $\Isom(X)$ acts cocompactly there are finitely many orbits of vertices. Pick a finite set $A$ which intersects each of these orbits.

  For the first statement we proceed by contradiction~: if this was not true, there would exist a sequence of isometries $f_n \colon B_X(x_n,n) \to X$ such that $f_n$ does never coincide on $B_X(x_n,r)$ with an element of $\Isom(X)$. By composing by suitable elements of $\Isom(X)$, we can assume that $x_n$ and $f_n(x_n)$ belong to $A$. So every $y \in X$ belongs to $B(x_n,n)$ for all $n$ large enough (say $n \geq n_y$). By taking a subsequence (diagonal argument) we can assume that the sequence $(f_n(y))_{n \geq n_y}$ is eventually stationary, for every $y$. Then $f(y)=\lim_n f_n(y)$ is a well-defined isometry of $X$, a contradition.

The second statement follows from the first. Indeed, assume that $R>r_2$ and that $Y$ is $R$-locally $X$. If $f\colon B_X(x,r_2)\to Y$ is an isometry, consider an isometry $g \colon B_Y(f(x),R) \to X$ and apply the first part of the Lemma to $g \circ f$~: there is an isometry $h$ of $X$ such that $g \circ f$ coincides with $h$ on $B_X(x,r)$. In particular $g^{-1} \circ h$, which is defined on $B_X(x,R)$, is an isometry $B_X(x,R) \to B_Y(f(x),R)$ which coincides with $f$ on $B_X(x,r)$.

For the third statement, consider a covering $p \colon X \to Y$. We have to prove that $B$, the set of all vertices $x \in X$ such that $p$ is injective on $B(x,R)$, is equal to $X$. We shall prove that (1) $B$ is nonempty and (2) if $x \in B$, then every neighbor of $x$ belongs to $B$. This will indeed imply that $B=X$ because $X$ is connected by convention.

First observe that, since  $p$ maps balls of radius $R$ onto balls of radius $R$, a vertex $x \in X$ belongs to $B$ if and only if $B(x,R)$ and $B(p(x),R)$ have the same cardinality.

To prove (1), consider $x_0 \in X$ a vertex minimizing the number of vertices in $B(x_0,R)$. $Y$ being $R$-locally $X$, $B(p(x_0),R)$ has the same cardinality as some $R$-ball in $X$, and in particular $|B(p(x_0),R)| \geq |B(x_0,R)|$. The reverse inequality holds because $p$ is a covering. Therefore, $|B(p(x_0),R)| = |B(x_0,R)|$ and $x_0 \in B$.

For (2), let $x \in B$ and $x'$ be a neighbor of $x$. The restriction of $p$ to $B(x',r_2) \subset B(x,R)$ is an injective covering, and hence is an isometry. By the second statement, it extends to an isometry $B(x',R) \to B(p(x'),R)$. In particular $B(x',R)$ and $B(p(x'),R)$ have the same cardinality and $x' \in B$.
\end{proof}

This lemma is the starting point of our approach for building a covering $X \to Y$ if $Y$ is $R$-locally $X$ in Theorem  \ref{thm:QT} and \ref{thm:mainIntro}. Indeed, we can start from an isometry $f_0 \colon B_X(x_0,R) \to Y$. By the Lemma if $d(x_0,x_1)\leq R-r_2$, we can define another isometry $B_X(x_1,R) \to Y$ that coincides with $f_0$ on $B_X(x_1,r)$. If we have a sequence $x_0,\dots,x_n$ in $X$ with $d(x_i,x_{i-1}) \leq R-r_2$, we can therefore define $f_i \colon B_X(x_i,R)\to Y$ such that $f_i$ and $f_{i-1}$ coincide on $B(x_{i},r)$. In this way, by choosing a path from $x_0$ to $x$ we can define an isometry $f_x\colon B(x,R) \to Y$ for each $x \in X$, but such a construction depends on the choice of the path. We will be able to make this idea work in two cases. The first and easiest case is when $X$ is a quasi-tree (Theorem \ref{thm:QT}), in which case we can define a prefered path between any two points. The second and harder case will be the situation in which $f_x$ does not depend on the path; it is Theorem \ref{thm:mainIntro}.

\begin{lem}\label{lem:asymp_dim_1} Let $X$ be a connected graph that is quasi-isometric to a tree. Then there exists $r_1>0$, a tree $T$, and a open covering $X = \cup_{u \in V(T)} O_u$ such that for each $u \neq v \in V(T)$,
\begin{itemize}
\item $O_u$ has diameter less than $r_1$ (for the distance in $X$).
\item $O_u \cap O_v \neq \emptyset$ if and only $(u,v)$ is an edge in $T$.
\end{itemize}
\end{lem}
\begin{proof}
Consider $V(X)$ the $0$-skeleton (the set of vertices) of $X$. There is a tree $T$ and a \emph{surjective} quasi-isometry $q \colon V(X) \to V(T)$ (see \cite{KM08} for an explicit construction). Extend $q$ to a continuous quasi-isometry $X \to T$, by sending an edge to the geodesic between the images by $q$ of its endpoints. Define $O_u$ as the preimage of $B_T(u,2/3)$ by $q$. We leave it to the reader to check the required properties.
\end{proof}
 
\begin{proof}[Proof of Theorem \ref{thm:QT}] Let $X$ be a connected graph that is quasi-isometric to a tree and with cocompact isometry group. Let $r_1$, $T$ and $(O_u)_{u \in T}$ be given by Lemma \ref{lem:asymp_dim_1}. 

Let $r \geq r_1$ and let $r_2$ be given by Lemma \ref{lem:cocompact} for this value of $r$.

We define $R = r+r_2$ and we will prove that $X$ is LG-rigid at scale $R$.

Let $Y$ be a space $R$-locally $X$. Let $O_u^{r_2} = \{x \in X, d(x,O_u) \leq r_2\}$ be the $r_2$-neighborhood of $O_u$. Our goal is to construct isometries $\phi_u \colon O_u^{r_2} \to Y$ such that for all $u,v \in V(T)$, 
\begin{equation}\label{eq:phi_u_compatibles} \phi_u\textrm{ and }\phi_v\textrm{ coincide on }O_u \cap O_v.
  \end{equation}This will prove the Theorem, since then the map $\phi$ defined by $\phi(x) = \phi_u(x)$ if $x \in O_u$ is a covering that is well-defined by \eqref{eq:phi_u_compatibles}.

Consider $S_0 = \{u\in V(T), B(x_0,r) \cap O_u \neq \emptyset\}$. Using that $B(x_0,r)$ is connected and that $O_u \cap O_v\neq \emptyset$ only when $u$ and $v$ are adjacent in $T$, we see that $S_0$ is connected. We take $(S_n)_{n \geq 0}$ an increasing sequence of connected subtrees of $T$ that covers $T$, such that $S_0=\emptyset$ and $S_n$ is obtained from $S_{n-1}$ by adding a vertex. We construct by induction maps $\phi_u$ for $u \in S_n$, that satisfy \eqref{eq:phi_u_compatibles} for all $u,v \in S_n$.

For $n=1$, $S_1 = \{u\}$. Since $O_u$ has diameter less than $r_1$, $O_u^{r_2}$ is contained in a ball $B(x_0,R)$ and we can define $\phi_u$ as the restriction to $O_u^{r_2}$ of any isometry from this ball to $Y$. If $n \geq 2$ and $S_{n} = \{v \} \cup S_{n-1}$, take $u\in S_{n-1}$ the unique vertex adjacent to $v$. To ensure that \eqref{eq:phi_u_compatibles} holds on $S_n$, we only have to construct $\phi_v \colon O_v^{r_2} \to Y$ that coincides with $\phi_u$ on $O_u \cap O_v$, because $O_v$ does not intersect $O_{u'}$ for the others $u' \in S_{n-1}$. Let $x \in O_u \cap O_v$. By Lemma \ref{lem:cocompact}, there is an isometry $\widetilde \phi\colon B(x,R) \to Y$ that coincides with $\phi_u$ on $B(x,r)$, and in particular on $O_u$ because $r_1 \leq r$. We define $\phi_v$ as the restriction of $\widetilde \phi$ to $O_v^{r_2}$, which makes sense because $O_v^{r_2} \subset B(x,R)$.
\end{proof}

\section{USLG-rigidity}\label{section:USLG}

The goal of this section is to study USLG-rigidity. If a graph $X$ is USLG-rigid at scales $(r,R)$, in particular two isometries of $X$ that coincide on a ball of radius $r$ must be equal. In other words the isometry group is discrete.  Theorem \ref{thm:mainIntro}, that we prove later in this section, is a reciprocal of this. Before that we notice that covers in USLG-rigid graphs have a very special form.

\begin{prop}\label{prop:USLG_implies_covers_are_quotients} If a vertex-transitive graph $X$ is USLG-rigid at scales $(r,R)$, then for every graph $Y$ that is $R$-locally $X$, there is a group $H$ acting freely by isometries on $X$ with the the following properties~:
  \begin{itemize}
  \item there is an isometry $H \backslash X \to Y$.
    \item the map $X \to H \backslash X$ is injective on balls of radius $R$.
  \end{itemize}
\end{prop}
Note that this last property is equivalent to $d(hx,x) \notin (0,2R]$ for every $x \in X$ and $h \in H$.
\begin{proof} Let $p \colon X \to Y$ be a covering as given by LG-rigidity. Define the group $H = \{g \in \mathrm{Aut}(X), p(gx) = p(x) \forall x \in X\}$. Clearly $p$ induces $H \backslash X \to Y$. Let us show that this map is injective. Let $x_1,x_2 \in X$. Assume that $p(x_1) = p(x_2) = y$. We want to find $g \in H$ such that $g x_1 = x_2$. Let $\psi \colon B_Y(y,R) \to X$ be an isometry. Using that $X$ is $R$-locally $X$ and that $X$ is USLG-rigid, and taking into account Corollary \ref{cor:selfcover} we see that there exist $g_1,g_2\in \mathrm{Aut}(X)$ which coincide with $\psi \circ p$ on $B_X(x_i,r)$. In particular $g=g_2^{-1} g_1$ in an element of $\mathrm{Aut}(X)$ such that $g x_1 = x_2$. To see that $g$ belongs to $H$ and conclude the proof of the proposition, notice that $p$ and $p \circ g$ are coverings of $Y$ by $X$ that coincide on $B_X(x_1,r)$. By the uniqueness of such a covering, $p = p\circ g$ as desired.
\end{proof}

We record here the following consequence of Proposition \ref{prop:USLG_implies_covers_are_quotients}, that will be used in Corollary \ref{cor:residualfinite}.
\begin{lem}\label{lem:WRF+USLG=>RF} Let $(\Gamma,S)$ be a Cayley graph which is USLG-rigid. If there exists a sequence of finite graphs $(Y_n)_{n \in \N}$ such that for every $n\in \N$, $Y_n$ is $n$-locally $(\Gamma,S)$, then $\Gamma$ is residually finite.
\end{lem}
\begin{proof} Let $0<r \leq R$ be such that $X =(\Gamma,S)$ is USLG-rigid at scales $(r,R)$.

To prove that $\Gamma$ is residually finite, for every finite set $F$ in $\Gamma$ we construct an action of $\Gamma$ on a finite set such that the elements in $F \setminus\{1_\Gamma\}$ have no fixed point. To do so take a finite set $F$ in $\Gamma$, and pick $n > R$ such that $F$ is contained in the ball of radius $2n$ around the identity in $(\Gamma,S)$. By the assumption there is a finite graph $Y$ that is $n$-locally $X$. Since $X$ is USLG-rigid at scales $(r,R)$ and $R<n$, by Proposition \ref{prop:USLG_implies_covers_are_quotients} there is a subgroup $H \subset \mathrm{Aut}(X)$ that acts freely on $X$ such that $Y$ identifies with $H \backslash X$. In particular the action of $\Gamma$ by right-multiplication on the vertex set of $X$ passes to the quotient $H \backslash X$, and non-trivial elements of length less than $2n$ in $\Gamma$ have no fixed point. In particular no element of $F\setminus \{1_\Gamma\}$ has a fixed point. This shows that $\Gamma$ is residually finite.
\end{proof}

\subsection{Proof of Theorem \ref{thm:mainIntro}}

Let $X$ be as in Theorem \ref{thm:mainIntro}. Let $k\geq 2$ such that $X$ is $k$-simply connected. Denote $G$ the isometry group of $X$. By the assumption that $G$ is discrete and cocompact, there exists $r_c\geq 0$ such that if two isometries $g$ and $g'$ in $G$ coincide on a ball $B_X(x,r_c)$ of radius $r_c$, then they are equal.

We shall prove the following precise form of Theorem \ref{thm:mainIntro}.
\begin{prop}\label{prop:USLG_rigid_strong_form} There exists $C>0$ such that $X$ is USLG-rigid at scales $(r,r+C)$ for every $r \geq r_c$.

\end{prop}
In the sequel, we let $G$ denote the isometry group of $X$.
We shall need the following lemma. 
\begin{lem}\label{lem1}
Given $r_1 \geq r_c$, there exists $r_2 \geq r_1$ such that the following holds:
\begin{itemize}
\item for every $x\in X$, the restriction to $B_X(x,r_1)$ of an isometry $f\colon B_X(x,r_2)\to X$ coincides with the restriction of an element of $G$;
\item the restriction to $B_X(x,r_1)$ of an isometry $f\colon B_X(x,r_2)\to X$ is uniquely determined by its restriction to $B_X(x,r_c)$.
 \end{itemize}
\end{lem}
\begin{proof}
The first part is Lemma \ref{lem:cocompact}.

For the second part, let $f,g \colon B_X(x,r_2) \to X$ be two isometries which coincide on $B_X(x,r_c)$. By the first part there exists $f',g' \in G$ which coincide with $f$ and $g$ respectively on $B_X(x,r_1)$. Since $f' = g'$ on $B_X(x,r_c)$, we get $f'=g'$, and in particular $f =g$ on $B_X(x,r_1)$.
\end{proof}
\begin{rem}\label{rem:slg+discrete->USLG} This lemma applied to $r_c+1$ provides us with $r_2^{(r_c+1)}$ such that if $Y$ is $r_2^{(r_c+1)}$-locally $X$ and $\phi_1,\phi_2 \colon X \to Y$ are covering maps that coincide on $B(x,r_c)$, then they coincide on $B(x,r_c+1)$ (and hence everywhere since $X$ is connected). This implies the following~: if we are able to prove that $X$ is USLG-rigid at some scales $(r,R)$ for $r\geq 1+r_c$, then $X$ is USLG-rigid at scales $(r_c+\delta,\max(R+\delta,r_2^{(r_c+1)}))$ for all $\delta \geq 0$. Indeed, if $\phi\colon B(x,\max(R+\delta,r_2^{(r_c+1)})) \to Y$ is an isometry, we can apply that $X$ is USLG-rigid at scales $(r,R)$ to the restriction of $\phi$ to $B(x',R)$ for every $x' \in B(x,\delta)$, and get a covering $\widetilde \phi_{x'}\colon X \to Y$ that coincides with $\phi$ on $B(x',r)$. If $x',x'' \in B(x,\delta)$ satisfy $d(x',x'')\leq 1$, the covering $\widetilde \phi_{x''}$ coincides with $\phi$ on $B(x'',r)$, and in particular on $B(x',r_c)$ because $r \geq r_c+1$. By our property defining $r_2^{(r_c+1)}$, we have $\widetilde \phi_{x''}= \widetilde \phi_{x'}$. Since $B(x,\delta)$ is connected we get that $\widetilde \phi_{x} = \widetilde \phi_{x'}$ for all $x' \in B(x,\delta)$, and in particular $\widetilde \phi_x$ coincides with $\phi$ on $B(x,r+\delta)$. This proves that there exists a covering $\widetilde \phi\colon X \to Y$ which coincides with $\phi$ on $B(x,r_c+\delta)$. It is the unique such, since it is the
unique covering that coincides with $\phi$ on the smaller ball $B(x,r_c)$.
\end{rem}


Take now $r_1 = r_c+t$ for some $t\geq 1$ to be determined later, and $r_2\geq r_1$ the radius given by Lemma \ref{lem1}. Let $Y$ be a
graph that is $R$-locally $X$  with $R \geq r_2+ t$. For every $x \in X$ denote by $\germ(x)$ the set of all isometries $\phi \colon B_X(x,r_1)\to Y$ that are restrictions of an isometry $B_X(x,r_2)\to Y$.

\begin{lem}\label{lem2} Let $x,x' \in X$ with $d(x,x') \leq t$ and $\phi \in \germ(x)$. Then there is one and only one element of $\germ(x')$ that coincides with $\phi$ on $B(x',r_c)$, and it coincides with $\phi$ on $B(x,r_1) \cap B(x',r_1)$.
\end{lem}
\begin{proof}
For the existence, by Lemma \ref{lem1} and the fact that balls of radius $R$ in $Y$ are isometric to balls of radius $R$ in $X$, $\phi \in \germ(x)$ is the restriction to $B(x,r_1)$ of (at least) one isometry $\widetilde \phi\colon B(x,R) \to Y$. Then the restriction of $\widetilde \phi$ to $B(x',r_2)$ is an isometry and hence defines an element of $\germ(x')$ that coincides with $\phi$ on $B(x,r_1) \cap B(x',r_1)$.

The uniqueness also follows from Lemma \ref{lem1}, which implies that every element of $\germ(y)$ is determined by its restriction to $B(y,r_c)$.
\end{proof}

\begin{prop}\label{prop:def_Fxx} Assume that $t \geq \frac k 2$. There is a unique family $(F_{x,x'})_{x,x'\in V(X)}$ where
\begin{enumerate}
\item\label{item1} $F_{x,x'}$ is a bijection from $\germ(x) \to \germ(x')$.
\item\label{item2} If $d(x,x') \leq t$ and $\phi \in \germ(x)$, then $F_{x,x'}(\phi)$ is the unique element of $\germ(x')$ that coincides with $\phi$ on $B(x',r_c)$.
\item\label{item3} $F_{x',x''} \circ F_{x,x'} = F_{x,x''}$ for all $x,x',x'' \in X$.
\end{enumerate}
\end{prop}
\begin{proof}
If $(x,x') \in X$ satisfy $d(x,x') \leq t$ Lemma \ref{lem2} provides a bijection 
\[F_{x,x'}^{(0)}\colon \germ(x)\to \germ(x')\]
satisfying (\ref{item2}).

For every sequence $(x_1,\dots,x_n)$ of vertices of $X$ where $d(x_i,x_{i+1}) \leq t$ we define $F_{(x_1,\ldots,x_n)}\colon \germ(x_1) \to \germ(x_n)$ by composing the bijections $F^{(0)}_{x_i,x_{i+1}}$ along the path. Then Lemma \ref{lem2} implies that
\begin{equation}\label{eq=truc} F_{(x_1,\dots,x_n)} = F^{(0)}_{x_1,x_n}\textrm{ if diam}(\{x_1,\dots,x_n\}) \leq t.\end{equation}

Let $\gamma \colon [0,1] \to X$ be a continuous path. For every subdivision $0=a_1 \leq a_2 \leq \dots a_n=1$ with $d(\gamma(a_{i+1}),\gamma(a_i))\leq t$, we can consider $F_{\gamma(a_1),\dots,\gamma(a_n)} \colon  \germ(\gamma(a)) \to \germ(\gamma(b))$, and by \eqref{eq=truc} $F_{\gamma(a_1),\dots,\gamma(a_n)}$ is unchanged if one passes to a finer subdivision, and hence does not depend on the subdivision. Denote this map by $F_\gamma \colon \germ(\gamma(a)) \to \germ(\gamma(b))$. 

But again by \eqref{eq=truc}, $F_\gamma$ is invariant under homotopy fixing the end points. Therefore the map $\gamma \mapsto F_\gamma$ induces a map on the fundamental groupoid $\Pi_1(X)$. By the definition of $F$, $F_\gamma$ is the identity of $\germ(\gamma(a))$ if and $\gamma(a) = \gamma(b)$ and $\mathrm{diam}(\gamma([a,b]) \leq t$. By the inequality $k \leq 2t$ and the fact that $X$ is $k$-simply connected, we get that $F_\gamma$ is the identity of $\germ(\gamma(a))$ for all paths $\gamma$ such that $\gamma(a) = \gamma(b)$. This implies that $F_{\gamma}$ depends only on the endpoints $\gamma(a)$ and $\gamma(b)$. We can define $F_{x,x'}$ as the common value of $F_{\gamma}$ for all such $\gamma$ with $\gamma(a)=x$ and $\gamma(b)=x'$, and the existence of $F$ satisfying \ref{item1},\ref{item2},\ref{item3} in Lemma is proved. The uniqueness is clear since $X$ is connected.
\end{proof}

We are ready to prove that $X$ is USLG-rigid. We now fix the value of $t$ to $t = \frac k 2$, so that $r_1 = r+\frac k 2$. Let $f\colon B_X(x_0,R) \to Y$ be an isometry. The restriction of $f$ to $B_X(x_0,r_1)$ defines $\phi_0 \in \germ(x_0)$. For every $x\in X$ we define $\phi_x = F_{x_0,x}(\phi_0)$ and $\pi(x) = \phi_x(x)$, where $(F_{x,x'})_{x,x'\in V(X)}$ is given by Proposition \ref{prop:def_Fxx}. Then by (\ref{item2}) in the Proposition (and Lemma \ref{lem2}), $\pi$ coincides with $\phi_x$ on $B_X(x,r_1)$ for every $x \in X$. In particular $\pi$ is a covering map and coincides with $f$ on $B(x_0,r)$. Let us prove the uniqueness of $\pi$. Let $\pi'$ be another such covering. In our vocabulary, the third part of Lemma \ref{lem:cocompact}, says that for every $x \in X$, the restriction of $\pi'$ to $B_X(x,r_1)$ belongs to $\germ(x)$. It follows from Proposition \ref{prop:def_Fxx} and induction on $d(x_0,x)$ that for every $x \in X$, $\pi'$ coincides with $\phi_x$ on $B_X(x,r_1)$. So $\pi'=\pi$. This proves that $X$ is USLG-rigid at scales $(r,R)$. This implies Proposition \ref{prop:USLG_rigid_strong_form} by Remark \ref{rem:slg+discrete->USLG}.

\section{Groups whose Cayley graphs all have discrete isometry group}\label{section:discrete}
We recall that we view the isometry group of a graph $X$ (and more generally every subgroup of it) as a topological group for the topology of pointwise convergence. We start by the following lemma.
\begin{lem}\label{lem:discreteimpliestorsionfree} An infinite finitely generated group with a non-trivial torsion element has a Cayley graph, the isometry group of which contains an infinite compact subgroup.
\end{lem}
\begin{proof}
Let $\Gamma$ be infinite and finitely generated, with finite symmetric generating set $S$. If $\Gamma$ is not torsion-free, it has a non-trivial finite subgroup $F$. Then $FSF = \{fsf'|f,f'\in F,s \in S\}$ is an $F$-biinvariant finite symmetric generating set and we claim that $(\Gamma,FSF)$ does not have a discrete isometry group. Indeed, any permutation of $\Gamma$ which preserves all left $F$-cosets is an isometry of $(\Gamma,FSF)$. This shows that the isometry group of $(\Gamma,FSF)$ contains the compact infinite group $\prod_{x \in \Gamma/F} \mathrm{Sym}(x)$, where $\mathrm{Sym}(x)$ is the group of permutations of the finite set $x$. 
\end{proof}
Oberve that if $\Isom(X)$ has an infinite compact subgroup, then $\Isom(X)$ cannot be discrete. So this lemma implies that a necessary condition on a finitely generated group $\Gamma$ to have all Cayley graphs with a discrete isometry group (or equivalenty for $\Gamma$ to be USLG-rigid by Theorem \ref{thm:mainIntro}) is that this group is torsion-free.

We will see in Corollary \ref{cor:discreteisometrygroup} that for a large class of groups (the groups appearing in Corollary \ref{thm:lattices}), being torsion-free is also a sufficient condition for all their Cayley graphs to have a discrete isometry group. 

In a slightly different direction (Proposition \ref{prop:discreteIsometriesCayleyGraph} and Theorem \ref{thm:discrete}) we prove that many groups admit a Cayley graph with discrete isometry group. 

Let us now turn our attention to the case of lattices in semisimple Lie groups and groups of polynomial growth. Our goal is to prove Corollary \ref{thm:lattices}. Let $\Gamma$ be as in Corollary \ref{thm:lattices}.  In order to apply Theorem \ref{thm:mainIntro}, one needs to show that the isometry group of any Cayley graph of $\Gamma$ is discrete.


We shall use the following easy fact, showing a converse to Lemma \ref{lem:discreteimpliestorsionfree}.

\begin{lem}\label{lem:nonfinitenormalsubgroup}
Let $\Gamma$ be an infinite, torsion-free finitely generated group, and let $S$ be a finite symmetric generating subset of $\Gamma$. Then the isometry  group of $X=(\Gamma,S)$ has no non-trivial compact normal subgroup.
\end{lem}
\begin{proof}
Let $G=\Isom(X)$, and assume for a contradiction that $G$ admits a non-trivial compact normal subgroup $K$. Then there exists a vertex $x$ whose $K$-orbit $Kx$ (which is finite by the definition of $K$ being compact) contains a vertex $y$ distinct from $x$. Since $\Gamma$ acts transitively, there exists $g\in \Gamma$ such that $gx=y$. Since $K$ is normalized by $g$, we deduce that $gKx=Kgx=Ky=Kx$. In particular, the orbit $\{ g^n x, n \in \Z\}$ is contained in the finite set. Since $\Gamma$ acts freely, this implies that $g$ has finite order: contradiction.
\end{proof}

Let us denote by $\C$ the class of finitely generated groups satisfying the following property: $\Gamma\in \C$ if every locally compact totally disconnected group $G$ containing $\Gamma$ as a uniform lattice has an open compact normal subgroup. Observe that a finitely presented group $\Gamma\in \C$ which is torsion-free is USLG-rigid by Theorem \ref{thm:mainIntro} and Lemma \ref{lem:nonfinitenormalsubgroup}. 

Recall the following result of Furman. 

\begin{thm}  \cite{Furman} \label{furman}
Let $\Gamma$ be an irreducible lattice in a connected semisimple real Lie group $G$ with finite center and no compact factor (in case $G$ is locally isomorphic to $\PSL(2,\R)$, we assume that $\Gamma$ is uniform). Let $H$ be a locally compact totally disconnected group such that  $\Gamma$ embeds as a lattice in $H$. Then there exists a finite index subgroup $H_0$ of $H$ containing $\Gamma$, and a compact open normal subgroup $K$ of $H_0$ such that $H_0/K\simeq \Gamma$. In particular, $\Gamma$ belongs to $\C$.
\end{thm}

Regarding groups with polynomial growth, we have the following result of Trofimov.

\begin{thm}\cite{T}\label{trofimov}
Let $X$ be a vertex-transitive graph with polynomial growth. Then its isometry group has a compact open normal subgroup. \end{thm}
Although the following is not required for the proof of Corollary \ref{thm:lattices}, we record it for the sake of completeness. 
\begin{cor}
 Finitely generated groups with polynomial growth belong to $\C$.
 \end{cor}
 \begin{proof} 
Let $\Gamma$ be a finitely generated group with polynomial growth and let $G$ be a  totally disconnected locally compact group containing $\Gamma$ as a uniform lattice. Since $\Gamma$
 is finitely generated, $G$ is compactly generated. Let $X$ be a Cayley-Abels graph for $G$ (see \S \ref{sec:Cayley-Abels}), and denote by $\phi: G\to \Isom(X)$ the action map. This morphism is continuous and proper, so that $\phi$ has compact kernel and closed cocompact image. The restriction of $\phi$ to $\Gamma$ has finite kernel, so that $\phi(\Gamma)$ is a uniform lattice in $\Isom(X)$. By the \v{S}varc-Milnor Lemma, $X$ has polynomial growth. By Trofimov's theorem we deduce that $\Isom(X)$ contains a normal compact open subgroup $U$. Hence $\phi^{-1}(U)$ is a compact open normal subgroup of $G$, so we are done.    
\end{proof}

Together with Lemma \ref{lem:nonfinitenormalsubgroup}, we obtain
\begin{cor}\label{cor:discreteisometrygroup}
Let $X$ be a Cayley graph of some finitely generated torsion-free group $\Gamma$ which either has polynomial growth, or is as in Theorem \ref{furman}. Then the isometry group of $X$ is discrete. 
\end{cor}
\begin{rem}  In \cite[Corollary 1.5]{Furman} this Corollary  for $\Gamma$ as in Theorem \ref{furman} was stated without the hypothesis that it is torsion-free. This hypothesis is necessary as explained in Lemma \ref{lem:discreteimpliestorsionfree}.
\end{rem}

We now end the proof of 
Corollary \ref{thm:lattices}. Let $X$ be a Cayley graph of a torsion-free finitely generated group $\Gamma$ with polynomial growth. Since $\Gamma$ is finitely presented, $X$ is large-scale simply connected. Moreover, by Theorem \ref{trofimov}, $\Isom(X)$ has a compact,
open, normal subgroup which is trivial by Lemma \ref{lem:nonfinitenormalsubgroup}. Hence $\Isom(X)$ is discrete, so $X$ is USLG-rigid by Theorem \ref{thm:mainIntro}. 

The proof for lattices in semisimple Lie groups is similar (recall that it is a classical fact that they are finitely presented, see for example \cite{Witte}). Let $G$ be a semisimple connected real Lie group with finite center $Z(G)$, and $\Gamma \subset G$ be an irreducible lattice. Recall that semisimple means that the Lie algebra is a finite direct sum of simple (without ideal) Lie algebras. By the torsion-free hypothesis $\Gamma$ does not intersect $Z(G)$, so is a lattice in $G/Z(G)$. We can therefore assume that $G$ is center-free. Then $G$ is the direct product of the center-free simple Lie groups corresponding to the simple factors $G_i$ of the Lie algebra, and similarly taking the quotient by the product of the compact factors we can assume that $G$ does not have compact factors. For completeness, let us also recall that the lattice being irreducible means here that its image in $\prod_{j \neq i} G_j$ is dense for every $i$. It remains to observe the cases that are excluded by Furman correspond to virtually free groups, \emph{i.e.} that non-uniform lattices in $\PSL(2,\R)$ (the only semisimple Lie group with trivial center and locally isomorphic $\PSL(2,\R)$) are virtually free. This is well-known, but we sketch a proof for completeness. By Selberg's theorem \cite{Sel}, up to taking a finite index subgroup, $\Gamma$ is torsion free, and therefore acts freely on the hyperbolic plane. The quotient of $\mathbb{H}^2/\Gamma$ is a non-compact hyperbolic surface, which is therefore homeomorphic to a punctured surface of genus $\geq 1$. It is then an easy exercise to check that its fundamental group, i.e.\ $\Gamma$ is a free non-abelian finitely generated group.

\section{Graphs that are not LG-rigid: Theorem  \ref{thm:2covering_notNormal2} for a direct product}\label{section:counterexampleGraph}

As explained in the introduction,  Theorem  \ref{thm:2covering_notNormal2} has a simpler proof in the case when $\widetilde H=H \times \Z/2\Z$. 
Let us first briefly explain its main ideas. First, the assumption that $H^2(G,\Z/2\Z)$ is infinite implies that there are infinitely many non-isomorphic central extensions of $G$ by $\Z/2\Z$. By choosing a generating set of $G$ and lifting it back to the corresponding central extension, we obtain infinitely many (a priori pairwise distinct) 2-sheeted graph coverings (for short: $2$-coverings) of the corresponding Cayley graphs of $G$.  Such a covering will be trivial (namely with two connected components) exactly when the central extension is trivial. The fact that for any $r>0$, there exists a non-trivial covering such that the ball of radius $r$ around the identity in $G$ has disconnected preimage ensures that there are infinitely many pairwise non isomorphic coverings.

So far we have a collection of Cayley graphs of central extensions of $G$,  the one corresponding to $\widetilde G$ being disconnected.
The second step consists in ``gluing" $2$-coverings of left-cosets of $G$ in $H$ in the following way: given an edge $S$ between two elements $g$ and $g'$ in $G$ that belong to different cosets of $H$, we throw in all four edges between the preimages of $g$ and $g'$. It will be enough for our purposes to take the trivial covering for all cosets except the one corresponding to $G$.
This gives us a $2$-covering\footnote{At this point ``covering" has a looser meaning: it only means that every element of $H$ has exactly two preimages. However, since the purpose of this paragraph is to explain the heuristic, we stick to this terminology for simplicity.} of $H$ (which is connected). If we pick a covering of $G$ that is trivial at large scale, then the resulting covering of $H$ coincides on a ball of large radius with the Cayley graph of $\widetilde H$ (Lemma \ref{lem:2covering_RlocallyX0}).  This provides us with an infinite collection of such coverings that are $r$-locally $\widetilde H$ for arbitrary large $r$. We are left with proving that none of  these $2$-coverings is isometric to $\widetilde H$. What turns out to be clear is that there are no such isometries that commute with the covering map (Lemma \ref{lem:no_isoX0Xq}). So an important part of the proof consists in choosing a special Cayley graph of $H$ which is sufficiently rigid so that isometries of any such 2-coverings of $H$ commute with the covering map. 

\

Before explaining this in detail, let us briefly recall the definition of $H^2(G,\Z/2\Z)$ and its connection with central extensions (see \cite{Br}).

Let $A$ be an abelian group (denoted additively). A central extension of a group $G$ (denoted multiplicatively) by $A$ is an extension
\[ 1\to A \to E \to G \to 1\]
where the image of $A$ lies in the center of $E$.  Let us recall that two extensions
\[ 1 \to A \xrightarrow{i_1} E_1 \xrightarrow{\tau_1} G \to 1,\ \ 1 \to A \xrightarrow{i_2} E_2 \xrightarrow{\tau_2} G \to 1\]
are called isomorphic if there is a group isomorphism $\varphi \colon E_1 \to E_2$ such that $\tau_2 \circ \varphi = \tau_1$ and $\varphi\circ i_1=i_2$ (note that the second condition follows from the first one when $A=\Z/2\Z$). Let us point out that, when the group of automorphisms of $A$ is trivial (for example when $A=\Z/2\Z$), all the extensions are central because the conjugation by an element of $E$ induces on $A$ a group automorphism.

Let us recall how the cohomology group $H^2(G,A)$ parametrizes the central extensions of $G$ by $A$. The group $H^2(G,A)$ is defined as the quotient of $Z^2(G,A)$, the set of functions $\varphi \colon G^2 \to A$ such that $\varphi(1,1)=0$ and $\varphi(g_1,g_2g_3)+\varphi(g_2,g_3) = \varphi(g_1g_2,g_3)+\varphi(g_1,g_2)$, viewed as an abelian group with pointwise operation, by its subgroup $B^2(G,A)$ of coboundaries, \emph{i.e.} maps of the form $(g_1,g_2)\mapsto \psi(g_1)+\psi(g_2)-\psi(g_1 g_2)$ for some function $\psi\colon G \to A$ such that $\psi(1)=0$. Every $\varphi \in Z^2(G,A)$ gives rise to a central extension 
\[ 1 \to A \to E \to G \to 1\]
together with a (set-theoretical) section $s\colon G \to E$ by setting $E = A \times G$ with the group operation $(a,g_1)(b,g_2) = (a+b+\varphi(g_1,g_2),g_1 g_2)$, and $s(h)=(0,h)$. Conversely every central extension $E$ of $G$ by $A$ and section $s$ of $G$ in $E$ give rise to an element of $Z^2(G,A)$, by setting $\varphi(g_1,g_2) = s(g_1) s(g_2) s(g_1 g_2)^{-1}$. Lastly two elements in $Z^2(G,A)$ give isomorphic extensions if and only if they differ by an element in $B^2(G,A)$.

\begin{lem}\label{lem:fromH2_to_coverings} Let $G$ be a group with a finite symmetric generating set $S$. If $H^2(G,\Z/2\Z)$ is infinite, there is a sequence of $2$-coverings $q_n\colon Y_n \to (G,S)$ such that $Y_n$ is connected but $q_n^{-1}(B_S(x,n))$ is disconnected for all $x \in G$.
\end{lem}
\begin{rem} Actually the graphs $Y_n$ in this Lemma are Cayley graphs of extensions of $G$ by $\Z/2\Z$.
\end{rem}
\begin{proof} First we claim that for all $n \geq 1$ there exists $\varphi_n \in Z^2(G,\Z/2\Z)$ which is not a coboundary and such that $\varphi_n(g_1,g_2) = 0$ if $|g_1|_S+|g_2|_S \leq n$. This follows from linear algebra considerations: $Z^2(G,\Z/2\Z)$ can be viewed as vector space over the field with $2$ elements, and our assumption that $H^2(G,\Z/2\Z)$ is infinite means that $B^2(G,\Z/2\Z)$ is an infinite codimensional subspace. It does therefore not contain the finite codimensional subspace made of the elements $\varphi \in Z^2(G,\Z/2\Z)$ that vanish on $\{(g_1,g_2), |g_1|_S+|g_2|_S \leq n\}$.

If $n \geq 2$ and $\varphi_n$ is as above, consider $E_n$ the central extension of $G$ by $\Z/2\Z$ constructed from $\varphi_n$ and define $S_n = \{(0,s), s \in S\}$. If $s \in S$, since $\varphi_n(1_H,s) = 0$, the unit of $E_n$ is $(0,1_G)$ and since $\varphi_n(s,s^{-1}) = 0$, we have that $(0,s)^{-1} = (0,s^{-1})$. The set $S_n$ is therefore a finite symmetric set in $E_n$, and the quotient map $q_n \colon E_n \to G$ induces a $2$-covering $q_n \colon Y_n \to (G,S)$. The assumption on $\varphi_n$ implies that $q_n^{-1}(B_S(1_G,n))$ is the disjoint union of $\{0\}\times B_S(1_G,n)$ and $\{1\}\times B_S(1_G,n) $; in particular it is disconnected. By transitivity $q_n^{-1}(B_S(x,n))$ is disconnected for all $x \in G$. To prove the lemma it remains to observe that $Y_n$ is connected because $\varphi_n$ is not a coboundary.
\end{proof}

Theorem \ref{thm:2covering_notNormal2} now follows from the more general proposition
\begin{prop}\label{prop:2covering_notNormalgraph} Let $G$ be a group with a finite symmetric generating set, and assume that there is a sequence of $2$-coverings $q_n \colon Y_n \to (G,S)$ satisfying the conclusion of Lemma \ref{lem:fromH2_to_coverings}. Then for every finitely presented group $H$ containing $G$ as a proper subgroup, there is a Cayley graph $X_0$ of $H \times \Z/2\Z $ that is not LG-rigid.
\end{prop}

To prove the Proposition, we complete $S$ into a finite generating set $T$ of $H$ by adding elements of $H\setminus G$ in a way that will be made precise in Lemma \ref{lem:choice_of_T}. This allows to identify the Cayley graph $(G,S)$ as a subgraph of the Cayley graph $(H,T)$. We measure the distortion of $(G,S)$ in $(H,T)$ by the function $\rho(R) = \sup\{ |g|_S \left|  g\in G, |g|_T \leq R \right. \}$.

Consider $X_0$, the Cayley graph of $H \times \Z/2\Z$ for the finite generating set 
\[T' = \{(1_H,1)\} \cup ( S \times \{0\}) \cup ( (T\setminus S) \times \{0,1\} ).\]
Observe that the subgraph with vertex set $G \times \Z/2\Z$ of $X_0$ is the union of two copies of $(G,S)$ where we added edges between pairs of same vertices.

Now if $q \colon Y \to (G,S)$ is another $2$-covering, we can get a new graph denoted $X_q$, by replacing $G \times \Z/2\Z$ inside $X_0$ by $Y$. This means that the vertex set of $X_q$ is the disjoint union of $(H \setminus G) \times \Z/2\Z$ and $Y$, equipped with the natural $2$-to-$1$ map $p \colon V(X_q) \to H$. Two vertices in $(H \setminus G) \times \Z/2\Z$ (two vertices in $Y$) are connected by an edge if they were connected by an edge in $X_0$ (respectively if there were connected by an edge in $Y$ or if they have the same image in $G$), and there is an edge between a vertex in $(H \setminus G) \times \Z/2\Z$ and a vertex in $Y$ if there was an edge between their images in $(H,T)$.

We denote by $\sim_q$ the equivalence relation on the vertex set of $X_q$ where $x\sim_q y$ if $p(x) = p(y)$.

We start by a lemma showing that for each $R>0$, $X_{q_n}$ is $R$-locally $X_0$ for $n$ large enough.
\begin{lem}\label{lem:2covering_RlocallyX0} Let $q\colon Y \to (G,S)$ be a $2$-covering and let $R\in \N$. If the graph $q^{-1}(B_S(x,\rho(2R)))$ is disconnected for all $x \in G$, then $X_q$ is $R$-locally $X_0$.
\end{lem}
\begin{proof} Consider a ball of radius $R$ in $X_q$. If it does not contain any vertex in $Y$, it is obvioulsy isometric to the corresponding ball in $X_0$. Otherwise it contains a point $y$ in $Y$, and is therefore contained in the ball $B$ of radius $2R$ around $y$. By the definition of $\rho$ the intersection of $B$ with $Y$ is contained in $q^{-1}(B_S(q(y),\rho(2R))$, which is two disjoint copies of $B_S(q(y),\rho(2R))$ by our assumption. This gives an isometry between the ball of radius $2R$ around $y$ in $X_q$ and a corresponding ball in $X_0$ and proves that $X_q$ is $R$-locally $X_0$.
\end{proof}
\begin{rem}\label{rem:R-locallyX0_over_H} The proof shows that there is an isometry from every ball of radius $R$ in $X_q$ to $X_0$ which sends $\sim_0$ to $\sim_q$.
\end{rem}

The next observation allows to distinguish in some weak sense the graphs $X_{q_n}$ and $X_0$.
\begin{lem}\label{lem:no_isoX0Xq} If $Y$ is connected, there is no isometry between $X_q$ and $X_0$ sending $\sim_q$ to $\sim_0$.
\end{lem}
\begin{proof}
Let us say that a subset $E$ of the edge set of $X_q$ is admissible if it has the property that for every vertex $x\in X_q$, every neighbor of $p(x)$ in $(H,T)$ has a preimage $y$ by $p$ such that $\{x,y\} \in E$. 

We claim that $X_0$ admits an admissible edge set which makes $X_0$ disconnected, but that $X_q$ does not admit such an admissible edge set if $Y$ is connected. This claim implies the Lemma because an isometry  between $X_q$ and $X_0$ sending $\sim_q$ to $\sim_0$ would send an admissible subset of edges to an admissible subset of edges.

The first claim is very easy, as we can just take for $E$ the set
\[ E = \{ \{(x,i),(y,j)\} \textrm{ edge of }X_0| i=j\}.\]
For the second claim, take an admissible edge subset $E$. Since $(H,T)$ is connected, every vertex in $X_q$ can be connected to an edge of $Y$ by a sequence of edges in $E$. Also, observe that if $\{x,y\}$ is an edge in $X_q$ that corresponds to an edge $S$ in $(H,T)$, \emph{i.e.} if $p(x)^{-1} p(y) \in S$, then $\{x,y\}$ is the only edge between $x$ and an element of $p^{-1}(p(y))$. This implies that $\{x,y\}\in E$ because $E$ is admissible. In particular $E$ contains all edges in $Y$. This shows that if $Y$ is connected, $X_q$ with edge set $E$ remains connected, as announced.
\end{proof}

The last step is to observe that for a well-chosen $T$, an isometry between $X_0$ and $X_q$ necessarily sends $\sim_0$ on $\sim_q$ (at least if $q$ has a large injectivity radius). The proof will rely on the notion of triangles in a graph. This notion will appear several times in later sections, in particular in the proof of Theorem  \ref{thm:discrete}, see also Lemma \ref{lem:Xtilde_remembers_Ytilde}.
\begin{defn}\label{def:triangles}
  A \emph{triangle} in a graph is a set consisting of $3$ vertices that pairwise connected by an edge.

  If $S$ is a finite symmetric generating subset of a group $G$, and $s \in S\setminus\{1_G\}$, we denote by $N_3(s,S)$ the number of triangles in the Cayley graph $(G,S)$ containing the vertices $1_G$ and $s$.
\end{defn}
  We start by 

\begin{lem}\label{lem:choice_of_T} Let $G=<S> \subsetneq H$ be as in Proposition \ref{prop:2covering_notNormalgraph}. There is a symmetric generating set $T \subset H$ such that $T \cap G=S$ and every isometry of $X_0$ preserves $\sim_0$, where $X_0$ is the Cayley graph of $H \times \Z/2\Z$ for the finite generating set 
\[T' = \{(1_H,1)\} \cup ( S \times \{0\}) \cup ( T\setminus S \times \{0,1\} ).\]
\end{lem}
\begin{proof} To lighten the notation, let us denote $T^*= T\setminus \{1_H\}.$

First pick an arbitrary finite symmetric generating set $T_1 \subset H$ such that $T_1 \cap G=S$. Let $M= \max_{t \in T_1} N_3(t,T_1)$. Observe that replacing $T_1$ by $T_1 \cup \{h,h^{-1}\}$ for $h \in H \setminus G$ of word-length $|h|_{T_1}>3$ does not change the function $N_3(\cdot,T_1)$ but increases the cardinality of $|T_1 \setminus G|$. Also, such an $h$ exists because our assumption on $G$ implies that $G$ is infinite, and therefore $H \setminus G$ also. Therefore there exists a finite symmetric generating subset $T \subset H$ such that $T \cap G=S$ and such that $\max_{t \in T^*} N_3(t,T) +1 < |T \setminus G|$.

On the other hand,  one checks that $N_3((1_H,1),T') = 2 |T \setminus S|$, whereas $N_3((t,\varepsilon),T') \leq 2+2 N_3(t,T)$ for every $(t,\varepsilon) \in T^* \times \{0,1\}$. The previous formula therefore implies that $N_3((1_H,1),T')> N_3(t',T')$ for every $t' \in T' \setminus \{(1_H,1),(1_{H},0)\}$. This means that the $\Z/2\Z$ cosets in $H \times \Z/2\Z$ are characterized in $X_0$ as the pairs of vertices that belong to exactly $2 |T\setminus S|$ triangles in $X_0$. The conclusion follows.
\end{proof}
 
We deduce by a straightforward compactness argument from the previous lemma that given some $r>0$,  there exists $R>0$ such that for every partial isometry between two balls $\phi:B(x,R)\to B(x',R)$, the restriction of $\phi$ to $B(x,r)$ preserves $\sim_0$. This implies the following 

\begin{cor}\label{cor:EquivalenceRelationLocallyRec}
  Let $G,H,S,T$ be as in Lemma \ref{lem:choice_of_T}. For all $r>0$, there exists $R>0$ such that for all $Y$ which is $R$-locally $X_0$, there exists a unique equivalence relation $\sim$ on the vertex set of $Y$ such that for all $x\in X$ and $y\in Y$, the restriction to $B(x,r)$ of some partial isometry $\phi:B(x,R)\to B(y,R)$ satisfies
  \[\phi(x_1) \sim \phi(x_2) \iff x_1 \sim_0 x_2\]
  for every  $x_1,x_2 \in B(x,r)$.
\end{cor}

We can now complete the proof of Proposition \ref{prop:2covering_notNormalgraph}. Consider $X_{q_n}$, the graph constructed from the $2$-covering $q_n \colon Y_n \to (G,S)$ given by the assumption of Proposition \ref{prop:2covering_notNormalgraph}, with $T$ given by Lemma \ref{lem:choice_of_T}. Lemma \ref{lem:2covering_RlocallyX0} implies that $X_{q_{\rho(2n)}}$ is $n$-locally $X_0$. Hence any covering map $\phi:X_0\to X_{q_{\rho(2n)}}$ must be injective in restriction to balls of radius $n$.

Observe that the preimages of the surjective graph morphism $p_n: X_{q_n}\to (H,T)$ have diameter $1$. It follows that $p_n$ is a $(1,1)$-quasi-isometry, so that by Theorem \ref{thm:QIinvariance}, there exists $k\in \N$ such that $ X_{q_n}$ is $k$-simply connected for all $n$.
Hence, by Proposition \ref{prop:equivalence_of_lssc}, for $n$ large enough, $\phi_n$ is an isometry. By Remark \ref{rem:R-locallyX0_over_H} and Corollary \ref{cor:EquivalenceRelationLocallyRec}, $\phi_n$ must send $\sim_0$ to $\sim_q$. This is a contradiction with Lemma \ref{lem:no_isoX0Xq}. This implies that $X_0$ is not LG-rigid and concludes the proof.

\section{Graphs that are not LG-rigid: Theorem \ref{thm:2covering_notNormal2} and \ref{thm:2covering_Normal2}}\label{section:counterexampleGraphNormal}

We now move to  Theorem \ref{thm:2covering_notNormal2} and \ref{thm:2covering_Normal2}, which will follow from the results in \S \ref{section:discreteisom} and from
\begin{thm}\label{thm=2coverings} Let $G \subset H$ be finitely generated groups, and $T$ a finite generating set of $H$ such that $S:=G \cap T$ generates $G$.
  
There exists $C \in \N$ such that the following holds. For every extension
\[1 \to \Z/2\Z \to G_{\tau} \xrightarrow{\tau} G \to 1\]
and symmetric subset $S_{\tau}\subset G_{\tau}$ such that $\tau$ maps $S_\tau$ bijectively onto $S$, we can associate a graph $X_{\tau}$ such that
\begin{enumerate}
\item\label{item=cayley} If $\tau$ comes from an extension
  \[1 \to \Z/2\Z \to \widetilde H \xrightarrow{\tau} H \to 1\]
  then $X_\tau$ is a Cayley graph of $\widetilde H$.
\item\label{item=4Lipschitz} For any two extensions $\tau,\tau'$ and any $S_\tau,S_{\tau'}$, the graphs $X_{\tau}$ and $X_{\tau'}$ are $4$-Lipschitz equivalent.
\item\label{item=Rmodeled} For every $R\in \R_+$, there exists $R_1\in \R_+$ such that for all  $(\tau,S_\tau)$ and $(\tau',S_{\tau'})$, the graph $X_{\tau'}$ is $R$-locally $X_{\tau}$ whenever the covering $(G_{\tau'},S_{\tau'}) \to (G,S)$ is $R_1$-locally\footnote{See Remark \ref{rem:thm}.} the covering $(G_{\tau},S_{\tau}) \to (G,S)$.
\item\label{item=automorphismgroup} If $\max_{t \in T}| tT \cap T| < |T|-|S|-1$, then the isometry group of $X_\tau$ is isomorphic as a topological group to an extension of a subgroup of the isometry group of $(H,T)$ by the compact group $(\Z/2\Z)^{H/G}$.
\item\label{item=finite_to_one} If $\max_{t \in T}| tT \cap T| < |T|-|S|-1$ and $(G,S)$ has a discrete isometry group, then the number of isomorphism classes of extensions $\tau'$ and $S_{\tau'}$ such that $X_{\tau'}$ is isometric to some given $X_{\tau}$ is at most $C$.
\item\label{item=transitive} If $G$ is normal in $H$ and $H$ splits as a semi-direct product $G \rtimes H/G$, then the isometry group of $X_\tau$ acts transitively.
\end{enumerate}
\end{thm}
\begin{rem}\label{rem:thm} In (\ref{item=Rmodeled}), we exceptionally allow a less restrictive notion of graph than in the rest of the paper, as we do not request that $S_\tau$ generates $G_\tau$. In that case $(G_\tau,S_\tau)$ is the disconnected graph without multiple edges nor loops with vertex set $G_\tau$ and with a vertex between $x,y$ if $x^{-1}y \in S_\tau$.

In (\ref{item=Rmodeled}) for a graph $Y$ and two coverings $q_1\colon \widetilde Y^{(1)} \to Y$ and $q_2\colon \widetilde Y^{(2)} \to Y$ we say that $q_1$ is $R_1$-locally $q_2$ if for every ball $B$ of radius $R_1$ in $Y$, there is an isometry $\phi$ between $q_1^{-1}(B)$ and $q_2^{(-1)}(B)$ such that $q_2 \circ \phi = q_1$.
\end{rem}

In particular, it follows from (\ref{item=cayley}) and (\ref{item=4Lipschitz}), and Theorem \ref{thm:QIinvariance} that there exists $k$ such that all $X_{\tau}$ are $k$-simply connected.

The rest of this section is devoted to the proof of this theorem, which is very similar to the proof of Proposition \ref{prop:2covering_notNormalgraph}, but involves significantly more work to ensure that items (\ref{item=finite_to_one}), (\ref{item=automorphismgroup}) and (\ref{item=transitive}) hold. For the trivial extension, $X_{\tau_0}$ coincides with the graph $X_0$ from \S \ref{section:counterexampleGraph}. For general $\tau$, the graph $X_{\tau}$ is obtained by copying above \emph{every} $G$-coset in $H$ a copy of the Cayley graph $(\widetilde G_{\tau},\widetilde S_{\tau})$, and adding in a suitable way edges (that we call \emph{outer} and \emph{vertical} edges) between different copies.  We first study this construction for general graphs, and then specialize to Cayley graphs.

\subsection{The construction in terms of graphs}
Let $X,Y$ be connected graphs, and assume that the vertex set of $X$ is partitionned as $X=\sqcup_{i\in I} Y_i$ into subgraphs that are each isometric to $Y$, and fix an isometry $f_i \colon Y \to Y_i$ for each $i \in I$. Assume that we are given a $2$-covering $q \colon \widetilde Y \to Y$. Note that $\widetilde Y$ does not need to be connected: in other words, the covering can be trivial. We define a graph $\widetilde X$ by putting over each $Y_i$ a copy $\widetilde Y_{i}$ of $\widetilde Y$, and connecting two vertices in $\widetilde Y_{i}$ and $\widetilde Y_{j}$ either if their images in $X$ are equal, or if $i \neq j$ and their images in $X$ are connected. Formally, the set of vertices of $\widetilde X$ is $\widetilde Y \times I$, and there are three types of edges: 
\begin{enumerate}
\item \emph{inner edges}: there is an edge between $(\widetilde y,i)$ and $(\widetilde y',i)$ if there is an edge between $\widetilde y$ and $\widetilde y'$ in $\widetilde Y$. 

\item \emph{vertical edges:} We put an edge betweeen $(\widetilde y,i)$ and $(\widetilde y',i)$ if $\widetilde y \neq \widetilde y'$ and $q( \widetilde y) = q(\widetilde y')$.

\item \emph{outer edges}: if $i \neq j$, there is an edge between $(\widetilde y,i)$ and $(\widetilde y',j)$ if and only if there is an edge in $X$ between $f_i(q (\widetilde y))$ and $f_j(q (\widetilde y'))$. 
\end{enumerate}

Then $\widetilde Y_{i}$ is $\widetilde Y \times \{i\}$, and there is a natural ``projection'' map $\widetilde X \to X$ sending $(\widetilde y,i)$ to $f_i(q(y))$.

We start by a lemma that will be used to show (\ref{item=4Lipschitz}) in Theorem \ref{thm=2coverings}. The rest of this subsection will be a series of Lemma studying the isometries of $\widetilde X$.

\begin{lem}\label{lem=2Lipschitz} If $\widetilde Y$, $\widetilde Y'$ are $2$-coverings of $Y$ and $\widetilde X,\widetilde X'$ are obtained by the above contruction, then any bijection $f \colon \widetilde X\to \widetilde X'$ which commutes with the projections $\widetilde X \to X$ and $\widetilde X' \to X$ is $2$-Lipschitz.
\end{lem}
\begin{proof} Let $\widetilde x_1$ and $\widetilde x_2$ be neighbors in $\widetilde X$. Let $x_1,x_2$ be their images in $X$, which by assumption are also the images of $f(\widetilde x_1), f(\widetilde x_2)$ by the projection $\widetilde X' \to X$. We have to show that $d(f(\widetilde x_1),f(\widetilde x_2)) \leq 2$. 

If $x_1=x_2$, then $f(\widetilde x_1)$ and $f(\widetilde x_2)$ are linked by a vertical edge: $d(f(\widetilde x_1),f(\widetilde x_2)) = 1$.

If $x_1 \neq x_2$, then the edge between $\widetilde x_1$ and $\widetilde x_2$ is an inner or an outer edge, and there is an edge between $x_1$ and $x_2$ in $X$. In particular $f(\widetilde x_1)$ has a least one neighbor $\widetilde  x' \in \widetilde X'$ (and two if the edge is an outer edge) that projects onto $x_2$. If $\widetilde x' = f(\widetilde x_2)$ then $d(f(\widetilde x_1),f(\widetilde x_2)) = 1$. Otherwise there is a vertical edge between $\widetilde x'$ and $f(\widetilde x_2)$ and $d(f(\widetilde x_1),f(\widetilde x_2)) = 2$.
\end{proof}

We will need a simple condition on $Y,X$  ensuring that the isometries of $\widetilde X$ commute with the projection $\widetilde X \to X$. This condition is in terms of triangles (see Definition \ref{def:triangles}). The condition is
\begin{equation}\label{eq:assumption_few_triangles} \textrm{Every edge in $X$ belongs to strictly less than $m_X-M_Y-1$ triangles,}\end{equation}
where $m_X$ is the minimal degree of $X$ and $M_Y$ the maximal degree of $Y$.

\begin{lem}\label{lem:Xtilde_remembers_Ytilde} Assume that \eqref{eq:assumption_few_triangles} holds. Then for every $2$-coverings $q_1\colon \widetilde Y^{(1)} \to Y$ and $q_2\colon \widetilde Y^{(2)} \to Y$ of $Y$ and every isometry $f\colon \widetilde X^{(1)} \to \widetilde X^{(2)}$, there is an isometry $g\colon X\to X$ which permutes the $Y_i$'s, and such that the projections $\widetilde X^{(1)} \to X$ and $\widetilde X^{(2)} \to X$ intertwine $f$ and $g$. 

In particular, if the graphs $\widetilde X^{(1)}$ and $\widetilde X^{(2)}$ are isometric, then the $2$-coverings are isomorphic: there are isometries $\phi \colon Y \to Y$ and $\widetilde \phi \colon \widetilde Y^{(1)} \to \widetilde Y^{(2)}$ such that $\phi \circ q_1 = q_2 \circ \widetilde \phi$.
\end{lem}
\begin{proof} Let $k=1$ or $2$. By construction, for every vertical edge between $(\widetilde y,i)$ and $(\widetilde y',i)$, in $\widetilde X^{(k)}$ there are at least as many triangles in $\widetilde X^{(k)}$ containing this edge as outer edges containing $(\widetilde y,i)$. This number is equal to twice the number of neighbors of $f_i(q_k(\widetilde y))$ in $X$ which are not in $Y_i$; in particular this number is at least $2(m_X-M_Y)$. On the other hand, the number of triangles containing an outer or inner edge is at most $2$ (a bound for the number triangles also containing a vertical edge) plus twice the number of triangles in $X$ containing the image of this edge. Hence by our assumption the number of triangles containing an outer or inner edge is strictly less than $2(m_X-M_Y)$. 

If  $f\colon \widetilde X^{(1)} \to \widetilde X^{(2)}$ is an isometry, it sends an edge to an edge belonging to the same number of triangles. By the preceding discussion it sends vertical edges to vertical edges. Therefore $f$ induces an isometry $g$ of $X$. It also sends bijectively outer edges to outer edges because the outer edges in $\widetilde X^{(k)}$ are the edges with the property that there are $3$ other edges in $\widetilde X^{(k)}$ corresponding to the same edge in $X$. This implies that $f$ preserves the partition of $X = \sqcup_{i\in I} Y_i$. Restricting $f$ to the any $\widetilde Y_i$ gives the desired isomorphism.
\end{proof}

The preceding lemma allows to describe the isometry group of $\widetilde X$ as an extension of a subgroup of the isometry group of $X$ by a compact group defined in terms of the Galois group of $q\colon \widetilde Y \to Y$, ie the group of automorphisms $\varphi$ of $\widetilde Y$ such that $q \circ \varphi = q$. Here $\widetilde Y$ is a $2$-covering of a connected graph, hence the Galois group is either $\Z/2\Z$ or trivial.
\begin{lem}\label{lem=automorphism_group} Assume that \eqref{eq:assumption_few_triangles} holds. Let $\widetilde Y$ be a $2$-covering of $Y$ and $\widetilde X$ obtained by the previous construction. 

If $f$ is an isometry of $\widetilde X$, there is a unique isometry $g$ of $X$ such that the projection $\widetilde X \to X$ intertwines $f$ and $g$. If we set $\pi(f) = g$, $\pi$ is a morphism from the isometry group of $\widetilde X$ to the isometry group of $X$ whose kernel is $F^{I}$, where $F$ is the Galois group of $\widetilde Y \to Y$.
\end{lem}
\begin{proof} First, there is a subgroup of the isometry group of $\widetilde X$ isomorphic to $F^{I}$, where $F^{I}$ acts by $(\varphi_{i})_{i \in I}\cdot (\widetilde y, j) = (\varphi_{j}(\widetilde y),j)$.

The existence of $g$ is Lemma  \ref{lem:Xtilde_remembers_Ytilde}, its uniqueness is clear, as is the fact that $\pi$ is a group morphism. It remains to understand the kernel of $\pi$. If $f$ belongs to the kernel of $\pi$, for every $i$ the restriction of $g$ to $\widetilde Y_{i}$ belongs to the Galois group of the cover $\widetilde Y_{i} \to Y_i$. This shows that the kernel of $\pi_0$ is contained in $(F^{I})^N$. The reverse inclusion is obvious. This shows the lemma.
\end{proof}

The last two lemmas isolate conditions on $X$ or on the $2$-covering $\widetilde Y \to Y$ that translate into transitivity properties of the graph $\widetilde X$. 
\begin{lem}\label{lem=transitive_on_I} Assume that \eqref{eq:assumption_few_triangles} holds. If there is a group $G$ acting transitively on $I$ and acting by isometries on $X$ such that $g \circ f_i = f_{gi}$ for all $g \in G,i \in I$, then there is a subgroup $G'$ in the isometry group of $\widetilde X$ such that $\pi(G') = G$ and such that each orbit of $\widetilde X$ under $G'$ meets each $\widetilde Y_{i}$.
\end{lem}
\begin{proof} For $g \in G'$, the map $(\widetilde y,i)\mapsto (\widetilde y,g i)$ is an isometry of $\widetilde X$, sends $\widetilde Y_{i}$ to $\widetilde Y_{gi}$ and belongs to $\pi^{-1}(g)$. One concludes by the assumption that the action of $G'$ on $I$ is transitive.
\end{proof}

\begin{lem}\label{lem=transitive_on_Yi} Assume that \eqref{eq:assumption_few_triangles} holds. Let $G_1$ be a group of isometries of $Y$ and $G_2$ a group of isometries of $X$ with the property that for all $i$ and all $g \in G_1$, there is an isometry $g' \in G_2$ of $X$ that preserves each $Y_j$, such that $f_j^{-1} \circ g' \circ f_j\in G_1$ for all $j$, and $f_i^{-1} \circ g' \circ f_i = g$.

Assume also that there exists a transitive group $\widetilde G_1$ of isometries of $\widetilde Y$ and a surjective group homomorphism $\widetilde G_1 \to G_1$ such that the covering $\widetilde Y \to Y$ intertwines the actions.

Then there is a subgroup $G'_2$ in the isometry group of $\widetilde X$ such that $\pi(G'_2) = G_2$ and which acts transitively on $\widetilde Y_{i}$ for each $i$.
\end{lem}
\begin{proof} Fix $(\widetilde y,i)$ and $(\widetilde y',i) \in \widetilde Y_{i}$. We construct an element of $\pi^{-1}(G_2)$ which sends $(\widetilde y,i)$ to $(\widetilde y',i)$. Since $\widetilde G_1$ acts transitively on $\widetilde Y$, there is $\widetilde g \in \widetilde G_1$ such that $\widetilde g \widetilde y = \widetilde y'$. Let $g$ be its image in $G_1$. By the first assumption there is an isometry $g' \in G_2$ that acts as an element $g_j$ of $G_1$ on each $Y_j$, and as $g$ on $Y_i$. Pick $\widetilde g_j \in \widetilde G_1$ in the preimage of the morphism $\widetilde G_1 \to G_1$, with $\widetilde g_{i} = \widetilde g$. Then the map $(\widetilde y,j) \mapsto (\widetilde g_j \widetilde y,j)$ is an isometry of $\widetilde X$ that preserves each $\widetilde Y_{j}$ and sends $(\widetilde y,i)$ to $(\widetilde y',i)$, as required. By construction it belongs to $\pi^{-1}(g')$.
\end{proof}

\subsection{The construction for Cayley graphs}\label{subsection=construction_cayleygraphs}
A particular case of this construction is the following situation. Let $H$ be a finitely generated group, with finite symmetric generating set $T$ not containing $1$. Let $G < H$ be a subgroup such that $S := T\cap G$ generates $G$. Take $X$ the Cayley graph $(H,T)$ and $Y$ the Cayley graph $(G,S)$. The partition of $H$ into left $G$-cosets gives a partition of $X$ into graphs isometric to $Y$, and every (set-theoretical) section $\alpha \colon H/G \to H$ gives rise to a family of isometries $(f_i \colon (G,S)\to (H,T))_{i \in H/G}$ given by $f_i(y) = \alpha(i) y$.

If $\{h,ht\}$ (for $h \in H$ and $t \in T$) is an arbitrary edge in $X$, the number of triangles in $X$ containing this edge is equal to the number of $h' \in H$ such that $h^{-1}h'$ and $t^{-1}h^{-1} h'$ belong to $T$, i.e.\ is equal to the cardinality of $tT \cap T$. Also, every edge in $X$ (respectively $Y$) has degree $|T|$ (respectively $|S|$). Therefore the condition \eqref{eq:assumption_few_triangles} holds if and only if $\max_{t \in T}|tT \cap T|<|T|-|S|-1$. 

We get a $2$-covering $q=q_{\tau}\colon \widetilde Y_\tau \to Y$ as above, for every extension 
\[ 1 \to \Z/2\Z \to G_{\tau} \xrightarrow{\tau} G \to 1\]
together with a symmetric subset $S_\tau \subset G_{\tau}$ mapping bijectively to $S$, by taking $\widetilde Y_\tau$ to be the Cayley graph $(G_{\tau},S_{\tau})$. 

\begin{rem}
Once again, we remark that $S_{\tau_0}=S \times \{0\}$ {\it is not} a generating subset of $G_{\tau_0}$, therefore $(G_{\tau_0},S_{\tau_0})$ is disconnected.
\end{rem}

Denote by $X_\tau$ the graph obtained from $q_\tau \colon \widetilde Y_\tau \to Y$ with the above construction.
\begin{lem}\label{lem:X_tau_Cayley} If $\tau$ is the restriction of an extension
  \[1 \to \Z/2\Z \xrightarrow{\iota} H_\tau \xrightarrow{\tau} H \to 1\]
  then $X_\tau$ is isometric to the Cayley graph of $H_\tau$ for the generating set
  \[ S_\tau \cup \tau^{-1}(T \setminus S) \cup \{ \iota(1)\}.\]
\end{lem}
\begin{proof} For every $i$, let $\widetilde h_i \in H_\tau$ such that $\tau(\widetilde h_i) = \alpha(i)$. The map $(\widetilde g,i) \in X_\tau \mapsto \widetilde h_i \widetilde g \in H_\tau$ is an isometry between $X_\tau$ and the Cayley graph of $H_\tau$ for the generating set
  \[ S_\tau \cup \tau^{-1}(T \setminus S) \cup \{ \iota(1)\}.\]
  Indeed, $S_\tau$ corresponds to inner edges, $\tau^{-1}(T \setminus S)$ to outer edges and $\iota(1)$ to vertical edges.
  \end{proof}

Let us assume that $\max_{t \in T}|tT \cap T|<|T|-|S|-1$. Then we can apply Lemma \ref{lem=automorphism_group}, \ref{lem=transitive_on_I} and \ref{lem=transitive_on_Yi}. This is the content of the next lemmas.

Let $\pi$ be the group morphism from the isometry group of $X_\tau$ to the isometry group of $X$ given by Lemma \ref{lem=automorphism_group}. We regard $H$ as a subgroup of the isometry group of $X$, acting by translation.
\begin{lem}\label{lem=pastedgraph_transitive} If $G$ is a normal subgroup and $H$ splits as a semi-direct product $G\rtimes H/G$, and if $\alpha$ is a group homomorphism, then $X_{\tau}$ is a transitive graph. More precisely, $\pi^{-1}(H)$ acts transitively on $X_\tau$.
\end{lem}
\begin{proof}
We first observe that there is a group $G'$ of isometries of $X_\tau$ acting transitively on each $\widetilde Y_{i}$ and such that $\pi(G')=G$. This follows from Lemma \ref{lem=transitive_on_Yi} and does not use that $H$ splits as a semi-direct product. 

Since $\alpha$ is a group homomorphism, we have that $\alpha(i) f_j(y) = f_{ij}(y)$ for all $y \in Y$ and $i,j \in H/G$. By Lemma \ref{lem=transitive_on_I} there is a group $G'_2$ of isometries of $X_\tau$ such that each $G'_2$-orbit meets each $\widetilde Y_{i}$, and such that $\pi(G'_2) = \alpha(H/G)$. 

The group generated by $G'$ and $G'_2$ therefore acts transitively on $X_\tau$, and its image by $\pi$ is the group generated by $G$ and $\alpha(H/G)$, which is $H$. This concludes the proof of the lemma.\end{proof}

\begin{lem}\label{lem=finite_to_one} Assume that the isometry group of $(G,S)$ is discrete. Let $G_{\tau},S_{\tau}$ be as above.

There are finitely many different isomorphism classes of extensions 
\[1 \to \Z/2\Z \to G_{\tau'} \xrightarrow{\tau'} G \to 1\]
and symmetric preimages $S_{\tau'} \subset G_{\tau'}$ of $S$ such that $X_{\tau}$ is isometric to $X_{\tau'}$.
\end{lem}
\begin{proof} By Lemma \ref{lem:Xtilde_remembers_Ytilde} we only have to prove that there are finitely many different isomorphism classes of extensions 
\[1 \to \Z/2\Z \to G_{\tau'} \xrightarrow{\tau'} G \to 1\]
and symmetric preimages $S_{\tau'} \subset G_{\tau'}$ of $S$ such that the resulting $2$-covering $(G_{\tau'},S_{\tau'}) \to (G,S)$ is isomorphic to $(G_{\tau},S_{\tau}) \to (G,S)$.

By definition, $(G_{\tau'},S_{\tau'}) \to (G,S)$ is isomorphic to $(G_{\tau},S_{\tau}) \to (G,S)$ if and only if there are isometries $\widetilde \phi \colon (G_{\tau'},S_{\tau'}) \to (G_{\tau},S_{\tau})$ and $\phi \colon (G,S) \to (G,S)$ such that $\tau \circ \widetilde \phi = \phi \circ \tau'$. Moreover since $G_{\tau'}$ acts transitively on $(G_{\tau'},S_{\tau'})$ we can always assume that $\widetilde \phi(1_{G_{\tau'}}) = 1_{G_{\tau}}$. In particular $\phi$ belongs to the stabilizer of the identity in the isometry group of $(G,S)$, which by assumption is finite. The Lemma therefore reduces to the observation that if $\phi$ is the identity, then $\widetilde \phi$ is a group isomorphism. Actually, $\widetilde \phi$ is even an isomorphism of rooted oriented marked Cayley graphs: since $\tau'$ and $\tau$ are bijections in restriction to $S_{\tau'}$ and $S_{\tau}$, we can label the oriented edges in $G_{\tau'}$ and $G_{\tau}$ by $S$, and the map $\widetilde \phi$ respects this labelling because $\phi =\mathrm{id}$ does.
\end{proof}

\subsection{Proof of Theorem \ref{thm=2coverings}} It remains to collect all the previous lemmas. Let $G,H,T$ be as in Theorem \ref{thm=2coverings}. If $H$ splits as a semidirect product $G \rtimes H/G$, there is a section $\alpha \colon H/G \to H$ that is a group homomorphism. Otherwise pick any set-theoretical section.

For every extension
\[ 1 \to \Z/2\Z \to G_{\tau} \xrightarrow{\tau} G \to 1\]
and a symmetric set $S_\tau \subset G_\tau$ such that $\tau$ is a bijection $S_\tau \to S$, we define $X_{\tau}$ as the graph defined in \S~\ref{subsection=construction_cayleygraphs} for this $\alpha$. 

(\ref{item=cayley}) has been proved in Lemma \ref{lem:X_tau_Cayley}, and (\ref{item=4Lipschitz}) in Lemma \ref{lem=2Lipschitz}. We leave to the reader the easy task to check (\ref{item=Rmodeled}), where $R_1$ is the maximum of $|g|_S$ over all $g \in G$ with $|g|_T \leq R$. Finally, (\ref{item=automorphismgroup}) is Lemma \ref{lem=automorphism_group}, (\ref{item=finite_to_one}) is Lemma \ref{lem=finite_to_one} and (\ref{item=transitive}) follows from Lemma \ref{lem=pastedgraph_transitive}.

\subsection{Concluding step in the proof of Theorem \ref{thm:2covering_notNormal2} and \ref{thm:2covering_Normal2}} 
We start by a proposition, the proof of which will be given in \S \ref{subsec:proof_discrete_groups}.
\begin{prop}\label{prop:discrete_isometry_group_enhanced} Let $G \subsetneq H$ be finitely generated groups, and assume that $G$ contains an element of infinite order. Then there is a finite symmetric generating set $T$ of $H\setminus \{1_H\}$ such that
\begin{itemize} 
\item The Cayley graph $(H,T)$ has a  discrete isometry group.
\item $S=G \cap T$ generates $G$ and the Cayley graph $(G,S)$ has a  discrete isometry group.
\item $\max_{t \in T}|tT \cap T|<|T| - |S|-1$.
\end{itemize}
\end{prop}

We also need the following Lemma.
\begin{lem}\label{lem:beacoup_dextensions} Let $G$ be a finitely generated group with finite generating set $S$. Let $R_1>0$. Assume that $H^2(G,\Z/2\Z)$ is infinite. Then for every extension
  \[ 1 \to \Z/2\Z \to G_0  \xrightarrow{\tau_0} G \to  1\]
where $S_0 \subset G_0$ is a symmetric subset such that $\tau_0$ is a bijection $S_0 \to S$, there is a  family $(\tau_i,S_i)_{i \in \R}$ where
\[ 1 \to \Z/2\Z \to G_i  \xrightarrow{\tau_i} G \to 1\]
are pairwise non isomorphic extensions, $S_i \subset G_i$ is a symmetric subset such that $\tau_i$ is a bijection $S_i \to S$, and where $(G_i,S_i)$ is $R_1$-locally $(G_0,S_0)$ for all $i$.
\end{lem}
\begin{proof}
It is easy to see that $H^2(G,\Z/2\Z)$ has the cardinality of the continuum. One way to argue is by using that an infinite compact Hausdorff topological group has always at least continuum many elements. In particular $H^2(G,\Z/2\Z)$, which is assumed to be infinite and which has a natural compact Hausdorff group topology as the quotient of the closed subgroup $Z^2(G,\Z/2\Z)$ of the compact Hausdorff group $(\Z/2\Z)^{G\times G}$ by its closed subgroup $B^2(G,\Z/2\Z)$, has (at least, but also clearly at most) the cardinality of the continuum. 

In particular by the same linear algebra consideration as in Lemma \ref{lem:fromH2_to_coverings} we see that there are continuum many elements $\varphi_i \in Z^2(G,\Z/2\Z)$ which are all distinct in $H^2(G,\Z/2\Z)$ and which vanish on $\{(g_1,g_2), |g_1|_S+|g_2|_S \leq R_1\}$. We conclude as in Lemma \ref{lem:fromH2_to_coverings}.
\end{proof}

Note that given an extension 
\[ 1 \to \Z/2\Z \to G_0  \xrightarrow{\tau_0} G \to  1\]
there is a symmetric subset $S_0 \subset G_0$ such that $\tau_0$ is a bijection $S_0 \to S$ if and only if every element of order $2$ in $S$ the preimages in $\widetilde H$ of every element of $G$ of order $2$ have order $2$. Therefore, when
\[ 1 \to \Z/2\Z \to G_0  \xrightarrow{\tau_0} G \to  1\]comes from an extension
\[ 1 \to \Z/2\Z \to \widetilde H  \to H \to  1,\]
the assumption that the preimages in $\widetilde H$ of every element of $G$ of order $2$ have order $2$ ensures that there is a symmetric subset  $S_0 \subset G_0$ such that $\tau_0$ is a bijection $S_0 \to S$.

It remains to combine this last Proposition and Lemma with  Theorem \ref{thm=2coverings} to conclude the proof of Theorem \ref{thm:2covering_Normal2}. Let $H,\tilde H,G$ as in Theorem \ref{thm:2covering_Normal2}. Let $T \subset H$ be the generating set provided by Proposition \ref{prop:discrete_isometry_group_enhanced}, $S=G\cap T$ and $C$ be the constant given by Theorem  \ref{thm=2coverings} for $G,H,T$. Let
\[ 1 \to \Z/2\Z \to G_0 \xrightarrow{\tau_0} G \to 1\]
be the restriction of
\[ 1 \to \Z/2\Z \to \tilde H \to H \to 1.\]
By the observation above, there is a symmetric subset $S_0\subset G_0$ mapping bijectively on $S$.

For $R>0$, let $R_1$ as in Theorem \ref{thm=2coverings}. Let $(\tau_i,S_i)_{i \in \R}$ be given by Lemma \ref{lem:beacoup_dextensions}. For every $i$ Theorem \ref{thm=2coverings} provides a graph $X_{\tau_i}$ such that
\begin{itemize}
\item for $i=0$, $X_{\tau_0}$ is a Cayley graph of $\widetilde H$.
\item for every $i,j$, $X_{\tau_i}$ and $X_{\tau_j}$ are $4$-Lipschitz equivalent.
\item for every $i$, $X_{\tau_i}$ is $R$-locally $X_{\tau_0}$.
\item for every $i$, there are at most $C$ different values of $j$ such that $X_{\tau_i}$ and $X_{\tau_j}$ are isometric.
\item the isometry group of $X_i$ is isomorphic to an extension of a subgroup of the (discrete) isometry group $(H,T)$ by $(\Z/2\Z)^{H/G}$.
\item if $G$ is normal and $H$ splits as a semidirect product, then $X_{\tau_i}$ is a transitive.
\end{itemize}

Theorem \ref{thm:2covering_notNormal2} follows, and Theorem  \ref{thm:2covering_Normal2} also once we justify that $H/G$ is infinite. But this holds because $H$ is finitely presented but $G$ is not.
\section{On Cayley graphs with discrete isometry group}\label{section:discreteisom}

This section is dedicated to the proofs of Theorems \ref{prop:discreteIsometriesCayleyGraph} and \ref{thm:discrete}. We start with a preliminary result dealing with marked Cayley graphs.

\subsection{The case of marked Cayley graphs}

For the proof of Theorem \ref{thm:discrete} and Theorem \ref{prop:discreteIsometriesCayleyGraph} we introduce the notion of \emph{marked Cayley graph}. If $\Gamma$ is a group with finite symmetric generating set $S$, the marked Cayley graph $(G,S)$ is the unoriented labelled graph in which each unoriented edge $\{\gamma,\gamma s\}$ is labelled by $\{s,s^{-1}\}$. With this notion, by an isometry of the marked Cayley graph $(\Gamma,S)$ we mean a bijection $f$ of $\Gamma$ such that $f(\gamma)^{-1}f(\gamma s) \in \{s,s^{-1}\}$ for all $s \in S$ and $\gamma \in \Gamma$.
\begin{lem}\label{lem:markedgraph} Let $\Gamma$ be a finitely generated group. There is a finite symmetric generating set $S$ such that the group of isometries of the marked Cayley graph $(\Gamma,S)$ is discrete.
\end{lem}
\begin{proof} Let $S_1$ be a symmetric finite generating set of $\Gamma$. Denote by $|\cdot|_1$ the word-length associated to $S_1$. Let $N$ be an integer strictly larger than the cardinality of $S_1$. Denote $S_N=\{\gamma \in \Gamma, |\gamma|_1 \in \{1,2,\dots,N\}\}$. We claim that the isometry group of the marked Cayley graph $(G,S_N)$ is discrete. For this we prove that an isometry $f$ of the marked Cayley graph $(G,S_N)$ that is the identity on the $|\cdot|_1$-ball of radius $N-1$ is the identity on $(G,S_N)$. We prove by induction on $n \geq N-1$ that $f$ is the identity on the $|\cdot|_1$-ball of radius $n$. Assume that the induction hypothesis holds for some $n\geq N-1$. 
Suppose for contradiction that there exists $|\gamma|_1=n+1$ such that $f(\gamma) \neq \gamma$. Then for every decomposition $\gamma = \gamma' s$ with $s \in S_N$ and $|s|_1+|\gamma'|_1 = n+1$, the fact that $f$ is an isometry of marked Cayley graph $(\Gamma,S_N)$ says that $f(\gamma) \in f(\gamma')\{s,s^{-1}\}$. By the induction hypothesis $f(\gamma')=\gamma'$, and 
$f(\gamma) =\gamma' s^{-1}$ 
because $f(\gamma) \neq \gamma$. Also $f(f(\gamma)) = \gamma$ because $f(f(\gamma))\in \gamma' \{s,s^{-1}\}$ and $f(f(\gamma)) \neq f(\gamma)$. 

Let us write $\gamma = \gamma_0 s_1 \dots s_N$ for $s_1,\dots,s_N \in S_1$ and $|\gamma_0|_1 = n+1-N$. Since $N > |S_1|$, there exists $k<l$ with $s_k=s_l$. By the preceding discussion for the decomposition $\gamma = (\gamma_0 s_1 \dots s_{k-1})(s_k \dots s_N)$, we obtain $f(\gamma) = \gamma_0 s_1 \dots s_{k-1} s_{N}^{-1} \dots s_k^{-1}$. By the same reasoning for the decomposition \[ f(\gamma) =( \gamma_0 s_1 \dots s_{k-1} s_{N}^{-1}\dots s_{l}^{-1})  (s_{l-1}^{-1} \dots s_k^{-1}),\] and using that $f(f(\gamma))=\gamma \neq f(\gamma)$, we have 
\[\gamma =  \gamma_0 s_1 \dots s_{k-1} s_{N}^{-1}\dots s_{l}^{-1} s_k \dots s_{l-1}.\] Since $s_k = s_{l}$, we obtain that $|\gamma|_1 \leq |\gamma_0|_1 + N-2 =n-1$, a contradiction. The map $f$ is therefore the identity on the $|\cdot|_1$-ball of radius $n+1$. This concludes the proof of the induction, and of the Lemma.
\end{proof}

\subsection{Proof of Theorem \ref{prop:discreteIsometriesCayleyGraph}}

By Lemma \ref{lem:markedgraph} there is a finite symmetric generating set $S_0$ of $\Gamma$ such that the marked Cayley graph $(\Gamma,S_0)$ has a discrete isometry group. For redactional purposes we also make sure that $1_\Gamma \notin S_0$.

Take $S$ a larger finite symmetric generating set containing $S_0$ but not $1_\Gamma$, with the property that for all $s \in S_0$, there exists $s' \in S$ such that $ss' \in S$ and $s \notin \{s',s'^{-1},ss',(ss')^{-1}\}$. Such an $S$ exists unless $\Gamma$ is finite, in which case there is nothing to prove.

Since $S$ contains $S_0$, the marked Cayley graph $(\Gamma,S)$ \emph{a fortiori} has a discrete isometry group.

We can assume that $S$ has at least three elements. Let $R$ be the number of vertices in the largest clique (=complete subgraph) in $(\Gamma,S)$. Decompose $S$ as a disjoint union $S=S_1\cup S_2 \cup S_2^{-1}$, where $S_1$ is the elements of $S$ of order $2$. Enumerate $S_1 \cup S_2$ as $s_1\dots,s_n$, with $n \geq 2$. Let $p_1,\dots,p_n$ be distinct integers, all strictly greater than $R$, and $F =\prod_{i=1}^n \Z/{p_i \Z}$, denoted additively. If the $p_i$ are prime, $F$ is a cyclic group. Consider the following symmetric generating set $\widetilde S$ of $\Gamma\times F$~:
\[ \widetilde S = \bigcup_{i=1}^n (\{s_i,s_i^{-1}\} \times \Z/p_i \Z) \cup \left(\{1_\Gamma\} \times (F\setminus \{0_F\})\right).\]
Let $X=(\Gamma \times F,\widetilde S)$, and $q\colon X \to (\Gamma,S)$ be the projection. For each $\gamma \in \Gamma$, $\{\gamma \} \times F$ is a clique with $|F|$ vertices, and observe that these are the only cliques with $|F|$ vertices. Indeed, let $K$ be a clique in $X$. Its image $q(K)$ is a clique in $(\Gamma,S)$, and therefore has cardinality at most $R$. By the fact that the preimage by $q$ of an edge in $(\Gamma,S)$ has cardinality at most $\max_i p_i$, we see that if $q(K)$ contains at least two points, then $K$ has cardinality at most $R \max_i p_i$, which is strictly less than $|F|$ because $n \geq 2$ and $R<\min_i p_i$.

Let $f$ be an isometry of $X$. It sends cliques to cliques, and therefore there is an isometry $f_0$ of $(\Gamma,S)$ such that $f_0 \circ q = q\circ f$. Since the number of edges between $\{\gamma\} \times F$ and $\{\gamma s\} \times F$ determines $\{s,s^{-1}\}$, we see that $f_0$ is an isomorphism of marked Cayley graphs. This defines a group homomorphism from the isometry group of $X$ to the isometry group of the marked Cayley graph $(\Gamma,S)$, which is discrete. To prove that the isometry group of $X$ is discrete we are left to prove that the kernel of this homomorphism is finite. Let $f$ be such that $f_0$ is the identity. This means that we can write $f(\gamma,x) = (\gamma,f_\gamma(x))$ for a family $f_\gamma$ of bijections of $F$. If $s \in S$, there is a unique $i$ such that $s \in \{s_i,s_i^{-1}\}$; denote by $F_s$ the subgroup $\Z/p_i\Z$ of $F$, so that there is an edge between $(\gamma,x)$ and $(\gamma s, x')$ if and only if $x-x' \in F_s$. In particular, there is an edge between $(\gamma,x)$ and $(\gamma s, x)$, and therefore also between their images by $f$. This means that $f_{\gamma s}(x) - f_\gamma(x) \in F_s$. Now take $s \in S_0$, and $s' \in S$ such that $ss' \in S$ and $s \notin \{s',s'^{-1},ss',(ss')^{-1}\}$, as made possible by our choice of $S$. Writing $f_{\gamma s}(x) - f_\gamma(x) =f_{\gamma s}(x) - f_{\gamma ss'}(x) +  f_{\gamma ss'}(x) - f_{\gamma}(x)$, we see that $f_{\gamma s}(x) - f_\gamma(x) \in F_s \cap (F_{s'} + F_{ss'}) = \{0\}$. This proves that for all $s \in S_0$ and $\gamma \in \Gamma$, $f_\gamma = f_{\gamma s}$. Since $S_0$ generates $\Gamma$, we have that $f_\gamma$ does not depend on $\gamma$. This proves that the set of isometries $f$ of $X$ such that $f_0$ is trivial is finite. This implies that the isometry group of $X$ is discrete, and proves Theorem \ref{prop:discreteIsometriesCayleyGraph}.

\subsection{Proof of Theorem \ref{thm:discrete}}

Let $\Gamma$ be a finitely generated group with an element of infinite order. By Lemma \ref{lem:markedgraph} there is a finite symmetric generating set $S_0$ of $\Gamma$ such that the marked Cayley graph $(\Gamma,S_0)$ has a discrete isometry group. Our strategy is to find a larger generating set $S$ such that we can recognize the marked Cayley graph $(\Gamma,S_0)$ from the triangles in $(\Gamma,S)$. For this, if $S$ is a symmetric subset of $\Gamma\setminus \{1\}$ and $s \in S$, recall that in Definition \ref{def:triangles} we have denoted by $N_3(s,S)$ the number of triangles in the Cayley graph $(\Gamma,S)$ containing the two vertices $e$ and $s$. In formulas, 
\[ N_3(s,S) = \left| \{ t \in S, s^{-1}t \in S\} \right|.\]
We will also denote $N_3(s,S)=0$ if $s \notin S$. By the invariance of $(\Gamma,S)$ by translations, $N_3(s,S)$ is also equal, for every $\gamma \in \Gamma$, to the number of triangles in $(\Gamma,S)$ containing the two vertices $\gamma$ and $\gamma s$. In particular, for $\gamma=s^{-1}$ we see that $N_3(s,S) = N_3(s^{-1},S)$. The main technical result is the following.
\begin{lem}\label{lem:increase_triangles} Let $S \subset \Gamma \setminus \{1\}$ be a finite symmetric set and $s_0 \in S$. There exists a finite symmetric set $S' \subset \Gamma \setminus\{1\}$ containing $S$ such that 
\begin{enumerate}
\item $S' \setminus S$ does not intersect $\{s^2,s \in S\}$.
\item $N_3(s,S') \leq 6$ for all $s \in S'\setminus S$.
\item\label{item:N3s} $N_3(s,S) = N_3(s,S')$ for all $s \in S \setminus \{s_0,s_0^{-1},s_0^2,s_0^{-2}\}$.
\item\label{item:N3s0} The ordered pair $(N_3(s_0,S')-N_3(s_0,S),N_3(s_0^2,S')-N_3(s_0^2,S))$ belongs to
\[ \left\{ \begin{array}{cc} 
\{(2,0),(4,0)\} & \textrm{if $s_0$ has order $2$}.\\
\{(1,1),(2,2),(3,3)\} & \textrm{if $s_0$ has order $3$}.\\
\{(1,0),(2,0),(2,2)\} & \textrm{if $s_0$ has order $4$}.\\
\{(1,0),(2,0),(2,1)\} & \textrm{if $s_0$ has order $\geq 5$}.
\end{array}\right.\]
\end{enumerate}
\end{lem}
\begin{proof} Let $\gamma \in \Gamma$ be an element of infinite order. We define a finite symmetric set by $S'=S\cup \Delta$ where $\Delta = \{\gamma^n, \gamma^{-n},s_0^{-1} \gamma^n, \gamma^{-n} s_0\}$ for an integer $n$ that we will specify later. Since all the $\gamma^n$ are distinct, for all $n$ large enough (say $|n| \geq n_0$) all the elements in $\Delta$ have word-length with respect to $S$ at least $3$, and the three elements $\gamma^n, \gamma^{-n},s_0^{-1} \gamma^n$ are distinct. This means that $\Delta$ has $4$ elements unless $s_0^{-1} \gamma^n= \gamma^{-n} s_0$, in which case $\Delta$ has $3$ elements.

Assume that $n \geq n_0$. Then the first condition clearly holds because an element of $\{s^2,s\in S\}$ has word length at most $2$, which is strictly smaller than $3$. Also, by the triangle inequality for the word-length with respect to $S$, a triangle in $(G,S')$ either is a triangle in $(G,S)$, or has at least two edges coming from $S' \setminus S=\Delta$. This shows the second item. Indeed, if $s \in S' \setminus S= \Delta$ and $t \in S'$ satisfies $s^{-1}t \in S'$, then either $t \in \Delta \setminus \{s\}$ or $s^{-1}t \in \Delta \setminus \{s^{-1}\}$, which leave at most $3+3=6$ possible triangles containing $e$ and $t$. This also shows that for $s \in S$,
\[ N_3(s,S') - N_3(s,S) = \left| \{ t \in \Delta, s^{-1}t \in \Delta\}\right| = \left| \Delta \cap s\Delta \right|.\]

It remains to find $|n| \geq n_0$ such that  (\ref{item:N3s}) and (\ref{item:N3s0}) hold.

Let us first consider the simpler case when there exists infinitely many $n$'s such that $s_0^{-1} \gamma^n= \gamma^{-n} s_0$. Then for such an $n$, $\Delta = \{\gamma^n, \gamma^{-n},s_0^{-1} \gamma^n\}$ and if $|n|\geq n_0$ the previous formula means that for $s \in S$, $N_3(s,S') - N_3(s,S)$ is the number of elements equal to $s$ in the list
\[ s_0,s_0^{-1}, \gamma^{2n}, \gamma^{-2n}, s_0^{-1}\gamma^{2n},\gamma^{-2n} s_0.\]
For $|n|$ large enough the terms $\gamma^{2n}, \gamma^{-2n}, s_0^{-1}\gamma^{2n},\gamma^{-2n} s_0$ do not belong to $S$, which proves that $N_3(s,S') - N_3(s,S) = 0$ if $s \notin\{s_0,s_0^{-1}\}$, and that $N_3(s_0,S') - N_3(s_0,S) \in \{1,2\}$ depending on whether $s_0$ has order $2$ or not. This proves (\ref{item:N3s}) and (\ref{item:N3s0}).

We now move to the case when $s_0^{-1} \gamma^n\neq \gamma^{-n} s_0$, \emph{i.e.} $\Delta$ has $4$ elements for all $|n|$ large enough. This means that $N_3(s,S') - N_3(s,S)$ is the number of elements equal to $s$ in the list
\[ s_0,s_0^{-1},\gamma^{-n} s_0 \gamma^{n},\gamma^{-n} s_0^{-1} \gamma^{n},\gamma^{-n} s_0 \gamma^{-n},\gamma^n s_0^{-1} \gamma^n,\gamma^{-n} s_0 \gamma^{-n}s_0,s_0^{-1}\gamma^n s_0^{-1} \gamma^n,\gamma^{2n}, \gamma^{-2n},s_0^{-1}\gamma^{2n},\gamma^{-2n}s_0.\]
If $n$ is large enough we can forget the last four elements, which do not belong to $S$. 

We have two actions of $\Z$ on $G$ given by $\alpha_n g=\gamma^{-n} g \gamma^{n}$ and $\beta_n g=\gamma^n g \gamma^{n}$. With this notation, the previous list becomes
\[ s_0,s_0^{-1},\alpha_n s_0,(\alpha_n s_0)^{-1},\beta_{-n} s_0,(\beta_{-n} s_0)^{-1}, (\beta_{-n} s_0) s_0,s_0^{-1}(\beta_{-n} s_0)^{-1}.\]
Denote by $T_1 \in \N \cup \{\infty\}$ and $T_2 \in \N \cup \{\infty\}$ the cardinality of the $\alpha$-orbit and the $\beta$-orbit of $s_0$ respectively, so that $\alpha_n s_0 = s_0$ if and only if $n$ is a multiple of $T_1$, and $\beta_n s_0 = s_0$ if and only if $n$ is a multiple of $T_2$ (with the convention that the only multiple of $\infty$ is $0$). If $n$ is a multiple of $T_1$ and $T_2$, then $\alpha_n s_0 = \beta_n s_0$, and hence $\gamma^{2n}=1$, which holds only if $n=0$. This implies that $T_1$ and $T_2$ cannot both be finite. Also, note that $T_2 <\infty$ prevents $s_0$ from having order $2$, because we assumed that $s_0^{-1} \gamma^n\neq \gamma^{-n} s_0$ for $n$ large enough.

Case 1: $T_1=T_2 = \infty$. Then all the terms in the previous list except $s_0,s_0^{-1}$ escape from $S$ as $n \to \infty$. This implies that for $n$ large enough $N_3(s,S') - N_3(s,S)=0$ if $s\notin \{s_0,s_0^{-1}\}$, and that $N_3(s_0,S')-N_3(s_0,S) \in \{1,2\}$ depending on whether $s_0$ is of order $2$. This proves (\ref{item:N3s}), and that $(N_3(s_0,S')-N_3(s_0,S),N_3(s_0^2,S')-N_3(s_0^2,S))$ is equal to $(2,0)$ if $s_0$ has order $2$, $(1,1)$ if $s_0$ has order $3$, and $(1,0)$ otherwise. This proves also (\ref{item:N3s0}).

Case 2: $T_1<\infty$, $T_2=\infty$. Take $n$ a large multiple of $T_1$. Then the terms containing $\beta_{-n} s_0$ in the previous list are not in $S$, and the elements of the list that can belong to $S$ are
\[ s_0,s_0^{-1},\alpha_n s_0=s_0,(\alpha_n s_0)^{-1} =s_0^{-1}.\]
This implies that $N_3(s,S') - N_3(s,S)=0$ if $s\notin \{s_0,s_0^{-1}\}$, and that $N_3(s_0,S')-N_3(s_0,S) \in \{2,4\}$ depending on whether $s_0$ is of order $2$. This proves (\ref{item:N3s}) and (\ref{item:N3s0}) as in the first case.

Case 3: $T_1=\infty$, $T_2<\infty$. Take $n$ a large multiple of $T_2$. Similarly the elements in the previous list that can belong to $S$ are
\[ s_0,s_0^{-1},s_0,s_0^{-1}, s_0^2,s_0^{-2}.\]
This proves (\ref{item:N3s}). If $s_0^2 \notin S$, by convention $N_3(s_0^2,S)=N_3(s_0^2, S')=0$, and we get as above that $N_3(s_0,S')-N_3(s_0,S) =2$ (recall that $s_0 \neq s_0^{-1}$ because $T_2<\infty$), which proves (\ref{item:N3s0}). If $s_0^2 \in S$, we get that  $(N_3(s_0,S')-N_3(s_0,S),N_3(s_0^2,S')-N_3(s_0^2,S))$ is equal to $(3,3)$ if $s_0$ has order $3$, and $(2,2)$ if $s_0$ has order $4$ and $(2,1)$ otherwise. This proves also (\ref{item:N3s0}). 
\end{proof}

We now prove
\begin{lem}\label{prop:recover_marked_graph_from_larger_graph} There exists a finite symmetric generating set $S \subset \Gamma \setminus \{1\}$ containing $S_0$ such that for every $s \in S_0$ and $s' \in S$, $N_3(s,S)=N_3(s',S)$ if and only if $s' \in \{s,s^{-1}\}$.
\end{lem}
Since an isometry of $(\Gamma,S)$ preserves the number of triangles adjacent to an edge, this proposition implies that the isometry group of $(\Gamma,S)$ is a subgroup of that of the marked Cayley graph $(\Gamma,S_0)$, which is discrete. This implies Theorem \ref{thm:discrete}. 
\begin{proof}[Proof of Lemma \ref{prop:recover_marked_graph_from_larger_graph}]
For a finite sequence $\underline u = u_1,\dots,u_N$ of elements in $S_0$, we define a finite symmetric generating sets $S(\underline u) \subset \Gamma \setminus \{1\}$ inductively as follows~: if $N=0$ (there are zero terms in the sequence), $S(\underline u) = S_0$, and if $N>0$ $S(\underline u)$ is the set $S'$ given by Lemma \ref{lem:increase_triangles} for $S=S(u_1,\dots,u_{N-1})$ and $s_0=u_N$.

By the first three items in Lemma \ref{lem:increase_triangles}, we have that $N_3(s,S(\underline u)) \leq 6$ for all $s \in S(\underline u) \setminus S_0$.

We claim that the conclusion of the Lemma holds for a good choice of $\underline u$. For this we consider $T_0=\emptyset \subset T_1 \subset \dots T_K=S_0$ a maximal strictly increasing sequence of symmetric subsets $T_i$ of $S_0$ with the property that for all $s \in S_0$, $s^2 \in T_i \implies s \in T_i$. We prove by induction on $i$ that there is a sequence $\underline u$ in $T_i$ such that for all $s,s' \in T_i$, $N_3(s,S(\underline u)) \geq 7$ and $N_3(s,S(\underline u))=N_3(s',S(\underline u))$ if and only if $s' \in \{s,s^{-1}\}$. For $i=0$ there is nothing to prove. Assume that there exists $\underline u$ in $T_i$ such that the conclusion holds for $T_i$. We will find a sequence $\underline{u'}$ in $T_{i+1}\setminus T_i$ such that for the concatenated sequence $\underline u,\underline{u'}$, the conclusion holds for $T_{i+1}$. Consider $t \in T_{i+1} \setminus T_i$. We consider two cases.

If $t^2\notin T_{i+1}$ or $t^2 = t^{-1}$, then by maximality, $T_{i+1} =T_i \cup \{t,t^{-1}\}$  (otherwise $T_{i+1}\setminus \{t,t^{-1}\}$ could be added between $T_i$ and $T_{i+1}$). We then define $\underline u'=t,\dots,t$ repeated $\max(n,7)$ times for $n > \max_{s \in T_i} N_3(s,S(\underline u))$, and we see that $N_3(s,S(\underline u,\underline u')) = N_3(s,S(\underline u))$ if $s \in T_i$ because $s \notin \{t,t^2,t^{-1},t^{-2}\}$, and $N_3(t,S(\underline u,\underline u')) \geq \max(n,7) > \max_{s \in T_i} N_3(s,S(\underline u,\underline u'))$. This proves the assertion for $T_{i+1}$.

If $t^2 \in T_{i+1}$ and $t^2 \neq t^{-1}$, observe that for all $j$, $t^{2^j} \in T_{i+1} \setminus T_i$ (otherwise if $j\geq 2$ is the smallest integer such that $t^{2^j} \notin T_{i+1} \setminus T_i$, then $t^{2^j} \notin T_i$ because it is the square of $t^{2^{j-1}} \notin T_i$, and hence $T_{i+1} \setminus \{t^{2^{j-1}},t^{-2^{j-1}}\}$ could be added between $T_i$ and $T_{i+1}$, contradicting the maximality). Since $T_{i+1}$ is finite, there is a smaller $j$ such that $t^{2^j} \in \cup_{k=0}^{j-1} \{t^{2^k},t^{-2^k}\}$, and by maximality necessarily $t^{2^j} \in \{t,t^{-1}\}$ and $T_{i+1}\setminus T_i = \{t^{2^k},k=0\dots j-1\} \cup  \{t^{-2^k},k=0\dots j-1\}$. In particular, $2^{2j}-1$ is a multiple of the order of $t$, which is therefore odd and hence at least $5$ (we assumed that $t^3 \neq e$). Take a sequence $n_0>n_1>\dots>n_j$, and take for $\underline u'$ the sequence containing $n_k$ times $t^{2^k}$ for all $k=0,\dots,j$. Then by (\ref{item:N3s}) $N_3(s,S(\underline u,\underline u')) = N_3(s,S(\underline u))$ if $s \in T_i$. Also, since by (\ref{item:N3s0}) for each occurence of $t^{2^j}$, $N_3(t^{2^j},\cdot)$ is increased by at least $1$ (and at most $2$), we see that $N_3(t^{\pm 2^j},S(\underline u,\underline u')) \geq n_j$, which can be made strictly larger than $\max_{s \in T_i} N_3(s,S(\underline u))$ and $7$ if $n_j$ is large enough. Finally, consider $k<j$. For each of the $n_k$ occurences of $t^{2^k}$ in $\underline u'$, only $N_3(t^{\pm 2^k},\cdot)$ and $N_3(t^{\pm 2^{k+1}},\cdot)$ can increase (by one or two), but necessarily $N_3(t^{\pm 2^k},\cdot)$ increases by at least one unit more than $N_3(t^{\pm 2^{k+1}},\cdot)$. This implies that
\begin{multline*} N_3(t^{\pm 2^k},S(\underline u,\underline u')) - N_3(t^{\pm 2^k},S(\underline u)) \\ \geq n_k +  N_3(t^{\pm 2^{k+1}},S(\underline u,\underline u')) - N_3(t^{\pm 2^{k+1}},S(\underline u)) - 2n_{k+1}.\end{multline*}
This implies that if $n_k$ is large enough compared to $n_{k+1}$, we have 
\[ N_3(t^{\pm 2^k},S(\underline u,\underline u')) > N_3(t^{\pm 2^{k+1}},S(\underline u,\underline u')).\]
In particular there is a choice of $n_0,\dots,n_j$ such that the induction hypothesis holds at step $i+1$.

Finally the induction hypothesis holds for $T_K=S_0$, which concludes the proof of the Lemma.
\end{proof}
\begin{rem} The proof of Theorem \ref{thm:discrete} in fact shows the following: there is a function $f \colon \N \to \N$ such that for any $N \in \N$, if $\Gamma$ is a group with $N$ generators and an element of order at least $f(N)$, then $\Gamma$ has a Cayley graph with discrete isometry group.
\end{rem}

\subsection{Proof of Proposition \ref{prop:discrete_isometry_group_enhanced}}\label{subsec:proof_discrete_groups}
We can adapt the proof of Theorem \ref{thm:discrete} to prove a slightly stronger statement: Proposition \ref{prop:discrete_isometry_group_enhanced} that was used in the proof of Theorem \ref{thm:2covering_Normal2}.

Let $G \subsetneq H$ be as in Proposition \ref{prop:discrete_isometry_group_enhanced}. It follows from Lemma \ref{lem:markedgraph} that $H$ has a finite symmetric generating set $T_0$ such that $S_0:=G\cap T_0$ generates $G$, and such that the isometry groups of the marked Cayley graphs $(G,S_0)$ and $(H,T_0)$ are discrete (just take for $T_0$ the union of a finite generating set of $G$ and of $H$ given by Lemma \ref{lem:markedgraph}). By applying the proof of Lemma \ref{prop:recover_marked_graph_from_larger_graph} first in $H$, we see that there is a finite symmetric generating set $T \subset H$ containing $T_0$ such that (1) $N_3(t,T) \leq 6$ for $t \in T \setminus T_0$, (2) if $t,t' \in T_0$, $N_3(t,T)=N_3(t',T)$ if and only if $t' \in \{t,t^{-1}\}$ and (3) $N_3(t,T) >6$ if $t \in T_0$. Now observe that adding to $T$ elements of $G\setminus T^2$ does not change the function $N_3(\cdot,T)$ on $H \setminus G$, whereas on $H$ it increases the functions $N_3(\cdot,T)$ and $N_3(\cdot,T \cap G)$ by the same amount. By applying the proof of Lemma \ref{prop:recover_marked_graph_from_larger_graph} to $G$, we therefore see that we can enlarge $T$ by adding elements of $G$ such that (1) (2) (3) still hold but also (2') if $s,s' \in S_0$, then $N_3(s,T \cap G)=N_3(s',T \cap G)$ if and only if $s' \in \{s,s^{-1}\}$ and (3') $N_3(s,T\cap G)>6$ for all $s \in S_0$. Finally, we observe that we can moreover assume that (4) $\max_{t \in T} |tT \cap T| <  |T \setminus G|-1$. This is because replacing $T$ by $T \cup \{h,h^{-1}\}$ for $h \in H \setminus G$ of word-length $|h|_T>3$ does not change the value of $\max_{t \in T} |tT \cap T|$ but increases the cardinality of $|T \setminus G|$; we can therefore repeat this as many times as necessary to ensure (4).

It follows from (1), (2) and (3) (respectively (1), (2') and (3')) that $(H,T)$ (respectively $(G,T \cap G)$) has a discrete isometry group. (4) is exactly the last point to be proved. This concludes the proof of Proposition \ref{prop:discrete_isometry_group_enhanced}.

\subsection{Proof of Corollary \ref{cor:residualfinite}} By Theorem \ref{thm:discrete}, $\Gamma$ has a Cayley graph $X$ with discrete isometry group. By Theorem \ref{thm:mainIntro}, $X$ is USLG-rigid. We conclude by Lemma \ref{lem:WRF+USLG=>RF}.

\section{Proof of Theorem \ref{thm:graph}}\label{section:counterexamplemetric}

\begin{lem}\label{lem:Cn} For each positive integer $n$, there exist geodesic contractible compact metric spaces $C_n^0,C_n^1,C_n^2$ with isometries $i_n^k$, $k=0,1,2$ from $[0,2^n]$ onto a subset $I_n^k  \subset C_n^k$ such that
\begin{itemize}
\item The isometry group of $C_n^k$ is trivial if $k=0,1$.
\item The isometry group of $C_n^2$ is isomorphic to $\Z/2\Z$ and acts as the identity on $I_n^2$.
\item For $k \neq l$, any two connected components of $C_n^k \setminus I_n^k $ and $C_n^l \setminus I_n^l$ are not isometric.
\item Every point in $C_n^k$ is at distance at most $2^{-n}$ from $I_n^k$, and every connected component of $C_n^k \setminus I_n^k$ contains a point at distance $2^{-n}$ from $I_n^k$.
\item For $k \neq l$ and every $x \in C_n^k$, there is an isometry from $B(x,2^{n-2}) \cup I_n^k$ to $C_n^l$ that maps $i_n^k(t)$ to $i_n^l(t)$ for all $t$.
\end{itemize}
\end{lem}
\begin{proof}
We start by constructing, for each integer $n \geq 1$, and each pair partition $\pi$ of $\{1,2,3,4,5,6\}$, a metric space $C_n^\pi$ as follows. We start from $6$ rectangles $[0,2^n] \times [0,2^{-n}]$, of length $2^n$ and height $2^{-n}$. We remove from the first and the third rectangles a ball of radius $3^{-n}$ and $4^{-n}$ respectively around the point $(2^{-n},2^{-n})$. We glue all the rectangles along the long edge $[0,2^n] \times\{0\}$. We also glue together the first and the second rectangles along the left segment $\{0\} \times [0,2^{-n}]$. We do the same for the third and fourth rectangles, and for the fifth and sixth rectangles. Finally for each class $\{i,j\}$ in the partition $\pi$, we glue together to right segments ${2^n}\times [0,2^{-n}]$ of the $i$-th and the $j$-th rectangle. The resulting space is $C_n^\pi$, that we equip with the unique geodesic metric that coincides with the euclidean metric on each (punctured) rectangle. See Figure \ref{picture=construction_Cn}.

\begin{figure}[!ht]
  \center
  \includegraphics{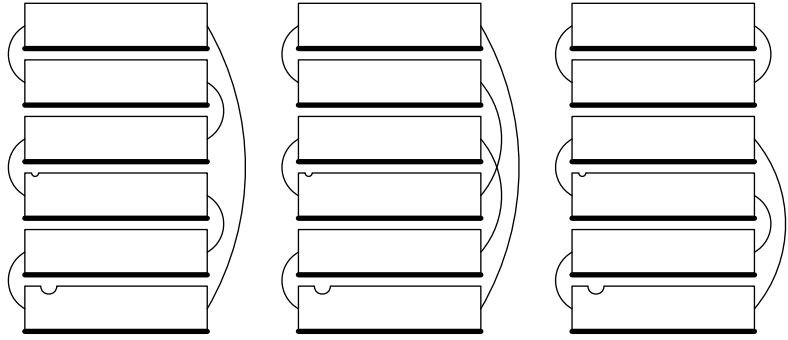}
  \caption{The spaces $C_n^0$ (left) and $C_n^1$ (middle) and $C_n^2$ (right), obtained by identifying the bottom side of all rectangles, and identifying each pair of vertical sides linked by an arc.}\label{picture=construction_Cn}
\end{figure}
Then one defines $C_n^0$ as $C_n^\pi$ for $\pi = \{\{1,6\},\{2,3\},\{4,5\}\}$, $C_n^1$ as $C_n^\pi$ for $\pi = \{\{1,6\},\{2,4\},\{3,5\}\}$ and $C_n^2$ as $C_n^\pi$ for $\pi = \{\{1,4\},\{2,3\},\{5,6\}\}$. By construction the exchange of the fifth and sixth rectangles gives an isometry of $C_n^2$. There is no difficulty in checking that there are no other non-trivial isometries, and that $C_n^0$ and $C_n^1$ have trivial isometry groups. The reason is that such an isometry must preserve the common long side of all the rectangles, and also the two small balls that have been removed, and hence must be the identity on the first and third rectangles. The rest of the properties are easy to check, with $I_n^k$ the identified edges $[0,2^n]\times \{0\}$ of the rectangles (the edge in bold in Figure \ref{picture=construction_Cn}). We only give a brief justification for the last one: a ball of radius $R<\frac{2^{n}-2^{-n} - 3^{-n}}{2}$ around a point in $C_n^\pi$ cannot simultaneously see one of the small balls that have been removed and a right side of a rectangle. The last point follows from the inequality $\frac{2^{n}-2^{-n} - 3^{-n}}{2} > 2^{n-2}$.
\end{proof}

Given the Lemma, we construct the space $X$ as follows. We start from a real line $\R$, and for each integer $n \geq 1$ and $m \in \Z$ we glue a copy of $C_n^0$ to $\R$ by identifying the segment $[m-2^{n-1},m+2^{n-1}]$ with $i_n^0([0,2^n])$ (through $t \mapsto i_n^0(t-m+2^{n-1})$. We equip $X$ with the unique euclidean metric that coincide with the metric on each copy of $C_n^0$. The properties (i) (ii) and (iii) are easy to verify from Lemma \ref{lem:Cn}, once we realize that we can recover $\R$ as the unique biinfinite geodesic in $X$.

Now for an arbitrary function $\sigma \colon \N \times \Z \to \{0,1,2\}$ we can modify the definition of $X$ by gluing to $[m-2^{n-1},m+2^{n-1}]$ a copy of $C_n^{\sigma(n,m)}$, to get a space $Y_\sigma$. Then the isometry group of $Y_\sigma$ is the semidirect product of $\prod_{m \in \Z} (\prod_{n, \sigma(m,n)=2} \Z/2\Z))$ by the subgroup of $\Z$ consisting of the elements $k$ satisfying $\sigma(m+k,n) =\sigma(m,n)$ for all $m,n$. Also $Y_R$ is $R$-locally $X$ if $\sigma(m,n)=0$ for all $(m,n)$ such that $2^{n-2}\leq R$. It is straightfoward that, taking appropriate choices for $\sigma$, we can find a continuum of non isometric metric spaces satisfying (iv) (respectively (v), respectively (vi))).

\appendix

\section{Uncountable second cohomology group $H^2(H,\Z/2\Z)$, by Jean-Claude Sikorav}\label{sec:appendix}

\subsection{Some reminders about group homology and cohomology (cf \cite{Br})}
Let $G$ be any group, and let $\F$ be a field. For a set $A$, denote by $\F^{(A)}$ the vector space over $\F$ with basis $A$ (the finitely supported functions from $A$ to $\F$). By definition, the {\it second homology group} $H_2(G,\F)$, actually a vector space over $\F$, is
$\frac{\ker\partial_2}{ {\rm im}\ \partial_3}$
where $\partial_2:\F^{(G^2)}\to \F^{(G)}$ and  $\partial_3:\F^{(G^3)}\to \F^{(G^2)}$ are defined by
\[\left\{\begin{array}{rl} \partial_2(g_1,g_2)&=(g_2)-(g_1g_2)+(g_1)\\
 \partial_3(g_1,g_2,g_3)&=(g_2,g_3)-(g_1g_2,g_3)+(g_1,g_2g_3)-(g_1,g_2).\end{array}\right.\]
The {\it second cohomology group} $H^2(G,\F)$ is  $\frac{\ker\delta^3}{{\rm im}\ \delta^2}$, with $\delta^2:\F^{G}\to \F^{G^2}$ and $\delta^3:\F^{G^2}\to \F^{G^3}$ the dual maps
\[\left\{ \begin{array}{rl} \delta^2c(g_1,g_2)&=c(g_2)-c(g_1g_2)+c(g_1)\\
\delta^3c(g_1,g_2,g_3)&=c(g_2,g_3)-c(g_1g_2,g_3)+c(g_1,g_2g_3)-c(g_1,g_2).\end{array}\right.\]
This implies that $H^2(G,\F)$ is isomorphic to the dual vector space $H_2(G,\F)^*$. 

\begin{cor} If $H_2(G,\F)$ is infinite-dimensional, then $H^2(G,\F)$ is uncountable (in fact has cardinality at least that of the continuum).
  \end{cor}

There is a topological interpretation which enables one to ``compute'' $H_2(G,\F)$ as follows. Let $X$ be a topological space whose fundamental group is $\pi_1(X)=G$. We make the following assumptions:
\begin{itemize}
  \item the homotopy groups $\pi_2(X)$ and $\pi_3(X)$ vanish; this is true in particular if $X$ is {\it aspherical} ie the universal covering $\widetilde X$ is contractible
\item $X$ is a cell complex (or CW-complex) ie $X=\displaystyle\bigcup_{n=0}^\infty X^{(n)}$, where $X^{(0)}$ is discrete and $X^{(n)}$ is obtained from $X^{(n-1)}$ by gluing a family of $n$-disks $D^n_i$, $i\in I_n$ along their boundaries, using some continuous attaching maps $f_i=\partial D^n_i\to X^{(n-1)}$: a point $x\in \partial D^n_i$ is identified with $f_i(x)\in X^{(n-1)}$.
The image in $X$ of a disk $D^n_i$ is called an {\it $n$-cell}. 
\end{itemize}

Then $\Gamma_n:=H_n(X^{(n)},X^{(n-1)};\F)$ has a basis $(e^n_i)_{i\in I_n}$ which can be identified with the $n$-cells, so that $\Gamma_n\approx \F^{(I_n)}$. And in the definition of $H_2(G,\F)$ one can replace $\F^{(G^n)}$, $n\le3$, by $\Gamma_n$,
the map $\partial_n$ being the connecting map in the homology sequence of the triple $(X^{(n)},X^{(n-1)},X^{(n-2)})$. 

Moreover, if $e^1_i$ is a $1$-cell with vertices $v_i^-$, $v_i^+$ ($0$-cells) and $x_i\in G$ is the associated generator of $G$, one has
$\partial_1(e^1_i)=v_i^+-v_i^-.$

\begin{rem} The group $G=\pi_1(X)$ has a presentation where the generators are the $1$-cells and the relations are associated the attaching maps of the $2$-cells. Conversely, every group $G$ with a presentation (finite or infinite) with ${\rm card}(I)$ generators and ${\rm Card}(J)$ relations is the fundamental group of an aspherical cell complex $X$, with the $2$-skeleton $X^{(2)}$ being associated to the presentation. Thus $X$ has one $0$-cell,  ${\rm card}(I)$ $1$-cells and ${\rm Card}(J)$ $2$-cells. If the number of relations is a finite number $q$, this implies
\[\dim H_2(G;\F)=\dim \frac{\ker\partial_2}{{\rm im}\ \partial_3}\le q<\infty.\]
Conversely, if $H_2(G;\F)$ is infinite-dimensional, there is no presentation of $G$ with a finite number of relations.
\end{rem}

\subsection{A criterion for uncountable $H^2$}

\begin{defn}\label{defn:finite_type} We say that {\it $G$ is of type $(p,q,r)$} if $G$ is the fundamental group of a  cell complex $X$ with $\pi_2(X)=\pi_3(X)=0$, and if $X$ has one $0$-cell, $p$ $1$-cells, $q$ $2$-cells and $r$ $3$-cells. This means that $G$ has a presentation with $p$ generators and $q$ relations, and at most $r$ ``relations between the relations''.
\end{defn}

\begin{prop} Let $u:G\to\Z$ be a nonzero group homomorphism. We assume that $G$ is of type $(p,q,r)$ with $q\ge p+r$.
Denote $\F=\Z/2\Z$ the field with two elements. Then 
\begin{enumerate}[(i)]
\item $H_2(\ker u;\F)$ is infinite dimensional over $\F$.
\medskip
\item $H^2(\ker u;\F)$ is uncountable.
\medskip
\item $\ker u$ is not finitely presented.
\end{enumerate}
\end{prop}

\begin{proof} We have seen that (i) implies (ii). By the remark, (i) implies (iii). Thus it suffices to prove (i).

  Let $X$ be as in Definition \ref{defn:finite_type}. Let $\widetilde X$ be the universal covering of $X$, which admits a free action such that the covering map $\widetilde X\to X$ is identified with $\widetilde X\to G\backslash\widetilde X$. Then $\widehat X:=\ker u\backslash\widetilde X$ is a covering of $X$ such that $\pi_1(\widehat X)=\ker u$. This covering is Galois with group ${\rm Aut}(\widehat X|X)= G/\ker u\approx\Z$. Assuming without loss of generality that $u$ is onto, ${\rm Aut}(\widehat X|X)$ is generated by an element $t=[g]$ such that $u(g)=1$. Then for every $g\in G$ its image in ${\rm Aut}(\widehat X|X)$ is $[g]=t^{u(g)}$.

  Each $n$-cell of $X$ lifts to a $\Z$-orbit of $n$-cells in $\widehat X$, thus
\[\Gamma_n=H_n(\widehat X^{(n)},\widehat X^{(n-1)};\F)\approx\F^{(\Z\times I_n)}\approx(\F[\Z])^{(I_n)}.\]
Moreover, the $\Z$-action on $\widehat X$ makes $\Gamma_n$ into a module over the group ring $\F[\Z]=\F[t,t^{-1}]$, and 
the isomorphism $\Gamma_n\approx\F[t,t^{-1}]^{(I_n)}$ is true not only as vector spaces over $\F$, but also as modules over $\F[t,t^{-1}]$. In particular
we can identify
\[\Gamma_3=(\F[t,t^{-1}])^r\ , \ \Gamma_2=(\F[t,t^{-1}])^q\ , \ \Gamma_1=(\F[t,t^{-1}])^p\ , \ \Gamma_0=\F[t,t^{-1}].\]
Also, the maps $\partial_n$ are $\F[t,t^{-1}]$-linear. Thus if we choose the $1$-cells in $\widehat X$ which form the basis of $\Gamma_1$ over $\F[t,t^{-1}]$ to start from the $0$-cell which is the basis of $\Gamma_0$, the property $\partial e^1_i=v_i^+-v_i^-$ implies
 \[\partial_1(\lambda_1,\cdots,\lambda_p)=\sum_{i=1}^p\lambda_i(t^{u(x_i)}-1)\]
where $x_1,\cdots,x_p$ are the generators of $G$ associated to the $1$-cells.

Denote $R=\F[t,t^{-1}]$, which is a principal ideal domain. We have thus a sequence of $R$-linear maps
\[R^r\xrightarrow{\partial_3} {}  R^q\xrightarrow{\partial_2} {}R^p\xrightarrow{\partial_1} {} R,\]
with $\partial_i\partial_{i+1}=0$, and $H_2(\ker u;\F)\approx\displaystyle{\frac{\ker\partial_2}{{\rm im}\ \partial_3}}$. Moreover, since $u(x_i)\ne0$ for some $i$, we have $\partial_1\ne0$. Since $R$ is a principal ideal domain, we have $\ker \partial_2 \approx R^k$, ${\rm im}\ \partial_3 \approx R^\ell$ for some $k,\ell$, and $\displaystyle{\frac{\ker \partial_2}{{\rm im}\ \partial_3}}\approx R^{k-\ell}\oplus T$ where $T$ is a torsion $R$-module. 
\smallskip
Since $R\approx\F^{(\Z)}$ as a $\F$-vector space, to finish the proof it suffices to show that $k-\ell>0$. For this, 
let $F=\F(t)$ be the fraction field of $R$. Consider the induced sequence of $F$-linear maps
\[F^r\xrightarrow{u_3} {}  F^q\xrightarrow{u_2} F^p\xrightarrow{u_1} F.\]
We have $u_iu_{i+1}=0$ and $u_1\ne 0$ thus $u_1$ is onto. Moreover, $\ker u_2\approx F^k$ and ${\rm im}\ u_3\approx F^\ell$, thus
\begin{eqnarray*}  k-\ell &=&\dim{\frac{\ker u_2}{{\rm im}\ u_3}}\\
&=&\dim{\frac{F^q}{{\rm im}\ u_3}}-\dim{\frac{F^q}{\ker u_2}}\\
&\ge&(q-r)-\dim{\rm im}\ u_2.\end{eqnarray*}
Finally, ${\rm im}\ u_2\subset\ker u_1$, which is of dimension $p-1$ since $u_1$ is onto. Thus
\[k-\ell\ge(q-r)-(p-1),\]
which is $>0$ by the hypothesis $q\ge p+r$, as required.
\end{proof}

\subsection{Examples with $\ker u$ finitely generated}
We now give explicit examples of finitely presented groups containing a finitely generated normal subgroup with uncountable second cohomology group with values in $\Z/2\Z$.
We use the following classical result : if $G=G_1\times G_2$  where $G_1$ and $G_2$ are finitely generated and $u:G\to \Z$ is nonzero on each factor, then $\ker u$ is finitely generated. 
It suffices to prove it for the product of two free groups $F(x_1,\cdots,x_n)\times F(y_1,\cdots,y_m)$: one reduces to the case when all generators are sent to $1$, then $\ker u$ is generated by 
\[(x_i,y_j)\ , \ 1\le i\le n, 1\le j\le m\ , \ (x_ix_1^{-1},1), 2\le i\le n\ , \ (1,y_jy_1^{-1}), 2\le j\le m.\]

1) Assume that $G_i$ is free with $p_i$ generators for some $p_i\ge2$. Then there exists a morphism $u:G\to\Z$ which is nonzero on each factor. Furthermore, we can take $X=X_1\times X_2$ where $X_i$ is a bouquet (or rose) of $p_i$ circles. Then $\widetilde X_i$ is a tree thus
$X_i$ associated to the presentation is aspherical, thus $X=X_1\times X_2$
is also aspherical. Thus $G$ is of type $(p,q,r)$ with
$$p=p_1+p_2\ , \ q=p_1p_2\ , \ r=0$$
Thus
$$q-(p+r)=p_1p_2-(p_1+p_2)=(p_1-1)(p_2-1)-1\ge0.$$

2) Assume that $G_i$ has a presentation with $p_i$ generators, $p_1\ge4,p_2\ge3$, and a unique relation which is primitive, for instance a surface group of genus $\ge2$. This implies first that there exists a morphism $u:G\to\Z$ which is nonzero on each factor. Furthermore, 
by \cite[Lemma p 382]{Cockcroft}, the $2$-complex $X_i$ associated to the presentation is aspherical, thus $X=X_1\times X_2$
is also aspherical. Thus $G$ is of type $(p,q,r)$ with
$$p=r=p_1+p_2\ , \ q=p_1p_2+2.$$
Thus
$$q-(p+r)=p_1p_2+2-2(p_1+p_2)=(p_1-2)(p_2-2)-2\ge0.$$
(We could also take the product of a free group with $p_1\ge2$ generators and a group with $p_2\ge3$ generators and one relation which is primitive; then $p=p_1+p_2,q=p_1p_2+1,r=p_1$ thus $q-(p+r)=(p_1-1)(p_2-2)-2$.)

\end{document}